\documentclass[noinfoline]{imsart}
\usepackage{amsmath,amssymb,amsthm,booktabs}
\usepackage{mathrsfs}
\usepackage{amsmath,amssymb,amsthm,enumerate,framed,graphicx,array,color,multirow,mathrsfs,amsrefs}

\usepackage[utf8]{inputenc}
\usepackage[colorlinks,citecolor=blue,urlcolor=blue]{hyperref}

\usepackage{bm}
\usepackage{euscript}
\usepackage{graphicx}

\usepackage{multicol}
\usepackage[usenames,dvipsnames,svgnames,table]{xcolor}

\usepackage{a4wide}

\usepackage{mathrsfs}
\usepackage{euscript}

\startlocaldefs
\numberwithin{equation}{section} \theoremstyle{plain}
\newtheorem{theorem}{Theorem}[section]
\newtheorem{lemma}{Lemma}[section]
\newtheorem{corollary}{Corollary}[section]
\newtheorem{proposition}{Proposition}[section]

\newtheorem{remark}{Remark}[section]
\endlocaldefs

\allowdisplaybreaks

\allowdisplaybreaks[4]
\numberwithin{equation}{section}

\def\bZ{\mathbb{Z}}

\def\bC{\mathbb{C}}
\def\bN{\mathbb{N}}
\def\bR{\mathbb{R}}
\def\bE{\mathbb{E}}
\def\bP{\mathbb{P}}

\def\bT{\mathrm{Tr}}
\def\br{\mathrm{rank}}

\def\e{\varepsilon}

\begin{document}

\begin{frontmatter}

  \title{Eigenvalue distributions of high-dimensional matrix processes driven by fractional Brownian motion}
    
  \runtitle{Empirical measures of matrices}

  \begin{aug}
    \author{\fnms{~ Jian} \snm{Song}\ead[label=e1]{txjsong@hotmail.com}}
    \and
    \author{\fnms{~ Jianfeng} \snm{Yao~}\ead[label=e2]{jeffyao@hku.hk}}
    \and
    \author{\fnms{~ Wangjun} \snm{Yuan}\ead[label=e3]{ywangjun@connect.hku.hk}}
    
    \affiliation{Shandong University and  The University of Hong Kong}
    \runauthor{J. Song,   J. Yao \& W. Yuan}

    \address{Research Center for Mathematics and Interdisciplinary Sciences, Shandong University, Qingdao, Shandong, 266237, China and School of Mathematics, Shandong University, Jinan, Shandong, 250100, China\\
      \printead{e1}}
    
    \address{
      Department of Statistics and Actuarial Science, 
      The University of Hong Kong\\
      \printead{e2}
    }

    \address{
      Department of Mathematics, 
      The University of Hong Kong\\
      \printead{e3}
    }
  \end{aug}

  \begin{abstract}
In this article, we study high-dimensional behavior of empirical spectral distributions $\{L_N(t), t\in[0,T]\}$   for a class of $N\times N$ symmetric/Hermitian random matrices,  whose entries are generated from the solution of stochastic differential equation driven by fractional Brownian motion with Hurst parameter $H \in(1/2,1)$.  For Wigner-type matrices, we obtain almost sure relative compactness of  $\{L_N(t), t\in[0,T]\}_{N\in\mathbb N}$   in $C([0,T], \mathbf P(\mathbb R))$ following the approach in \cite{Anderson2010}; for Wishart-type matrices, we obtain tightness of $\{L_N(t), t\in[0,T]\}_{N\in\mathbb N}$ on $C([0,T], \mathbf P(\mathbb R))$ by tightness criterions provided in Appendix \ref{subset:tightness argument}.  The limit of $\{L_N(t), t\in[0,T]\}$ as $N\to \infty$ is also characterised.  \end{abstract}



\end{frontmatter}

{
  \hypersetup{linkcolor=black}
  \tableofcontents
}

\section{Introduction} \label{sec-Intro}

There  has been {  increasing  research activity} on matrix-valued
stochastic differential equations (SDEs) in recent years. A prominent
example is the class of generalized Wishart processes introduced in
\cite{Graczyk2013}. This class generalizes classical examples of
symmetric Brownian motion \cite{Dyson62}, the Wishart process
\cite{Bru91}  and the symmetric matrix-valued process whose entries are independent Ornstein-Uhlenbeck processes \cite{Chan1992}. Recent works on generalized Wishart processes include \cite{Song20,Song19} and \cite{Graczyk2014,Graczyk2019}. A common feature of these processes is that they are all driven by independent Brownian motions.

In contrast, the study of SDE matrices driven by fractional Brownian motions (fBm) has a shorter history and less literature. To our best knowledge, the first paper is \cite{Nualart20144266}, where the symmetric fractional Brownian matrix was studied. The SDE for the associated eigenvalues and conditions for non-collision of the eigenvalues were obtained when the Hurst parameter $H > 1/2$ by using fractional calculus and Malliavin calculus. The convergence in distribution of the empirical spectral measure of a scaled symmetric fractional Brownian matrix was established in \cite{Pardo2016} by using Malliavin calculus and a tightness argument. These results were generalized to centered Gaussian process in \cite{Jaramillo2019}. Besides, \cite{Pardo2017} obtained SDEs for the eigenvalues, conditions for non-collision of eigenvalues and the convergence in law of the empirical eigenvalue measure processes for a scaled fractional Wishart matrix (the product of an independent fractional Brownian matrix and its transpose). In the present paper, by  adopting a different approach, we obtain stronger results on the convergence of empirical measures of eigenvalues for a significantly larger class of matrix-valued processes driven by fBms (see Remark \ref{remark-6}).

More precisely, consider the following 1-dimensional SDE
\begin{align} \label{SDE}
	dX_t = \sigma(X_t) \circ dB_t^H + b(X_t)dt,\quad t \ge 0,
\end{align}
with initial value $X_0$ independent of $B^H=(B^H_t)_{t\ge0}$. Here, $B_t^H$ is a fractional Brownian motion with Hurst parameter $H \in (1/2,1)$, and the differential $\circ dB_t^H$ is in the {\em Stratonovich} sense. It has been shown in \cite{Lyons1994} that there is a unique solution to SDE \eqref{SDE} if the coefficient functions $\sigma$ and $b$ have bounded derivatives which are H\"{o}lder continuous of order greater than $1/H-1$.

Let $\{X_{ij}(t)\}_{i,j \ge 1}$ be i.i.d. copies of $X_t$ and $Y^N(t) = \left(Y_{ij}^N(t) \right)_{1 \le i,j \le N}$ be a symmetric $N \times N$ matrix with entries 
\begin{align} \label{matrix entries Wigner}
	Y_{ij}^N(t) =
	\begin{cases}
	\dfrac{1}{\sqrt{N}} X_{ij}(t), & 1 \le i < j \le N, \\
	\dfrac{\sqrt{2}}{\sqrt{N}} X_{ii}(t), & 1 \le i \le N.
	\end{cases}
\end{align}
Let $\lambda_1^N(t) \le \cdots \le \lambda_N^N(t)$ be the eigenvalues of $Y^N(t)$ and
\begin{align}\label{eq-em}
	L_N(t)(dx) = \dfrac{1}{N} \sum_{i=1}^N \delta_{\lambda_i^N(t)}(dx)
\end{align}
be the empirical distribution of the eigenvalues.

In this paper, we aim to study, as the dimension $N$ tends to
infinity, the limiting behavior of the empirical measure-valued
processes $\{L_N(t), t\in[0,T]\}$ defined by \eqref{eq-em} and the
empirical measure-valued processes arising from other related
matrix-valued processes in the space $C([0,T], \mathbf P(\bR))$ of
continuous measure-valued processes, with $\mathbf P(\bR)$ being the
space of probability measures equipped with its weak topology.
{  A summary of the contents of the paper is as follows.}

In Section \ref{sec:Holder}, we study the solution of  SDE \eqref{SDE}. In \cite{Hu2007}, a pathwise upper bound for $|X_t|$ on the time interval $[0,T]$ was obtained under some boundedness and smoothness conditions on the coefficient functions. We adapt the
techniques used in \cite{Hu2007} to study the increments $|X_t - X_s|$ for $t, s \in [0,T]$ and obtain the pathwise H\"{o}lder continuity for
$X_t$. Moreover, under some integrability conditions on the initial
value $X_0$, the H\"{o}lder norm of $X_t$ is shown to be
$L^p$-integrable.  These results improve the previous results
  in \cite{nr02} with  sharper bounds for the increments of the process.

In Section \ref{sec:hdl}, we prove the {\em almost sure}
high-dimensional convergence in $C([0,T], \mathbf P(\bR))$ for the
empirical spectral measure-valued process \eqref{eq-em} of the
Wigner-type matrix \eqref{matrix entries Wigner}. We also obtain the
PDE for the Stieltjes transform $G_t(z)$ of the limiting measure
process $\mu_t$, which { turns to be}  the complex Burgers' equation up to a change of variable (see Remark \ref{remark-1'} and Remark \ref{reamrk-3'}). This generalizes the results of \cite{Pardo2016}, where the special case of $X_t = B_t^H$ was studied (note that the convergence obtained therein is in law). A complex analogue is also studied at the end of this section.

In Section \ref{sec:dependent}, we extend the model of Wigner-type
matrix to the locally dependent symmetric matrix-valued stochastic
process $R^N(t) = \left( R^N_{ij}(t) \right)_{1 \le i, j \le N}$,
where each $R_{ij}^N(t)$ is a weighted sum of
i.i.d. $\{X_{(k,l)}(t)\}$ for $(k,l)$ in a fixed and bounded
neighborhood of $(i,j)$ (see \eqref{matrix entries dependent} for the
definition). We also establish the {\em almost sure} convergence of
the empirical spectral measure-valued process in $C([0,T], \mathbf
P(\bR))$. Moreover, the limiting measure-valued process $\mu_t$ is
characterized by an equation satisfied by its Stieltjes transform (see
Theorem \ref{Thm-LED dependence}).
{ It is worth noticing that for the  proofs of}
 the {\em almost sure} convergence for the empirical spectral
 measure-valued processes in both Section \ref{sec:hdl} and Section
 \ref{sec:dependent}, we follow the strategy used in \cite{Song20} (which was inspired by \cite{Anderson2010}).

In Section \ref{sec:Wishart}, we study the high-dimensional limit of the empirical spectral measure-valued processes of the Wishart-type matrix-valued processes given by \eqref{matrix entries Wishart} and its complex analogue \eqref{matrix entries Wishart-complex}. In the special case where $X_t = B_t^H$, we recover the convergence result in \cite{Pardo2017}. We also obtain the PDE for the Stieltjes transform $G_t(z)$ of the limit measure process $\mu_t$ (see Remarks \ref{remark-2'} and \ref{remark-4'}).

For  Wishart-type matrix-valued processes, we are not able to obtain the {\em almost sure} convergence in $C([0,T], \mathbf P(\bR))$ for the empirical spectral measure-valued processes as done in Section \ref{sec:hdl} and Section \ref{sec:dependent} since the method used there heavily relies on the independence of the upper triangular entries.  Instead, we obtain the convergence in law by a tightness criterion for probability measures on $C([0,T], \mathbf P(\bR))$ (see Theorem \ref{Thm-0'},  Propositions \ref{Thm-1'} and \ref{remark:b1} in Appendix \ref{subset:tightness argument}).  The tightness criterion Theorem \ref{Thm-0'} might be already well-known in the literature (we refer to, e.g.,  \cite{Rogers1993, Pardo2016, Pardo2017, Jaramillo2019} for related arguments), however, we could not find a reference stating it explicitly, and we provide it in Appendix \ref{subset:tightness argument}.

 Finally, for the reader's convenience, some preliminaries on random matrices used in the proofs are provided in Appendix \ref{sec:pre-matrix}, and some useful lemmas and tightness criterions are provided in  Appendix \ref{subset:tightness argument}.

\section{H\"{o}lder continuity of the solution to the SDE~\protect\eqref{SDE}} \label{sec:Holder}

Some notations are in order.
The H\"{o}lder norm of a H\"{o}lder continuous function $f$ of order $\beta$ is
\begin{align*}
	\|f\|_{a,b;\beta} = \sup_{a \le x < y \le b} \dfrac{|f(x) - f(y)|}{|x-y|^{\beta}}.
\end{align*}
We also use 
\begin{align*}
	\|M\|_F = \left( \sum_{i,j=1}^N |M_{ij}|^2 \right)^{1/2}.
\end{align*}
to denote  the Frobenius norm (also known as the Hilbert-Schmidt norm or $2$-Schatten norm) of a $N \times N$ matrix $M = (M_{ij})_{1 \le i, j \le N}$.

\subsection{Preliminaries on fractional calculus and fractional Brownian motion}

In this subsection, we recall some basic results in  fractional calculus. See \cite{skm93} for more details.
Let $a,b\in \mathbb{R}$ with $a<b$ and let $ \alpha >0.$  The
left-sided  and right-sided  fractional Riemann-Liouville integrals
of $f\in L^1([a,b])$ of order $\alpha$   are   defined for almost all $t\in(a,b)$ by
\[
I_{a+}^{\alpha }f\left( t\right) =\frac{1}{\Gamma \left( \alpha \right) }%
\int_{a}^{t}\left( t-s\right) ^{\alpha -1}f\left( s\right)
ds\,,
\]%
and
\[
I_{b-}^{\alpha }f\left( t\right) =\frac{\left( -1\right) ^{-\alpha
}}{\Gamma \left( \alpha \right) }\int_{t}^{b}\left( s-t\right)
^{\alpha -1}f\left( s\right) ds,
\]%
respectively, where $\left( -1\right) ^{-\alpha }=e^{-i\pi \alpha }$ and $%
\displaystyle\Gamma \left( \alpha \right) =\int_{0}^{\infty }r^{\alpha
-1}e^{-r}dr$ is the Euler gamma function.

Let $I_{a+}^\alpha(L^p)$ (resp. $I_{b-}^\alpha(L^p)$) be the image of $L^p([a,b])$ by the operator $I_{a+}^\alpha$ (resp. $I_{b-}^\alpha$).  If $f\in I_{a+}^\alpha(L^p)$ (resp. $f\in I_{b-}^\alpha(L^p)$ ) and $\alpha\in(0,1)$, then the left-sided and right-sided  fractional derivatives are defined as
\begin{align}
&D_{a+}^{\alpha }f\left( t\right) =\frac{1}{\Gamma \left( 1-\alpha \right) }%
\left( \frac{f\left( t\right) }{\left( t-a\right) ^{\alpha }}+\alpha
\int_{a}^{t}\frac{f\left( t\right) -f\left( s\right) }{\left(
t-s\right) ^{\alpha +1}}ds\right), \notag\\
&D_{b-}^{\alpha }f\left( t\right) =\frac{\left( -1\right) ^{\alpha
}}{\Gamma \left( 1-\alpha \right) }\left( \frac{f\left( t\right)
}{\left( b-t\right)
^{\alpha }}+\alpha \int_{t}^{b}\frac{f\left( t\right) -f\left( s\right) }{%
\left( s-t\right) ^{\alpha +1}}ds\right)  \label{e-fd}
\end{align}
respectively, for almost all $t\in(a,b)$.

Let $C^\alpha([a,b])$ denote the space of $\alpha$-H\"older continuous functions of order $\alpha$ on the interval $[a,b]$. When $\alpha p>1$, then we have $I_{a+}^\alpha(L^p)\subset C^{\alpha-\frac1p}([a,b])$. On the other hand, if $\beta>\alpha$, then $C^\beta([a,b])\subset I_{a+}^\alpha(L^p)$ for all $p>1$.

The following inversion formulas hold:

\begin{alignat*}{3}
&I_{a+}^\alpha(I_{a+}^\beta f)=I_{a+}^{\alpha+\beta} f, & \qquad~& f\in L^1;\\
&D_{a+}^\alpha(I_{a+}^\alpha f)=f, & ~~&f\in L^1; \\
&I_{a+}^\alpha(D_{a+}^\alpha f)=f, & ~~&f\in I_{a+}^\alpha(L^1);\\
&D_{a+}^\alpha(D_{a+}^\beta f)=D_{a+}^{\alpha+\beta} f, &~~ &f\in I_{a+}^{\alpha+\beta}(L^1) , ~\alpha+\beta\le1.
\end{alignat*}
Similar inversion formulas hold for the operators $I_{b-}^\alpha$ and $D_{b-}^\alpha$ as well.

We also have the following integration by parts formula.
\begin{proposition}\label{p0}
If $f\in I_{a+}^\alpha(L^p),~g\in I_{b-}^\alpha(L^q)$ and $\frac1p+\frac1q=1$, we have
\begin{equation}\label{ibp}
\int_a^b (D_{a+}^\alpha f)(s)g(s)ds=\int_a^b f(s)(D_{b-}^\alpha g)(s)ds.
\end{equation}
\end{proposition}
The following proposition indicates the relationship between Young's integral and Lebesgue integral.
\begin{proposition}
\label{p1} Suppose that $f\in C^{\lambda }(a,b)$ and $g\in C^{\mu
}(a,b)$ with $\lambda +\mu >1$. Let ${\lambda }>\alpha $ and $\mu
>1-\alpha $. Then the Riemann-Stieltjes integral $\int_{a}^{b}fdg$
exists and it can be
expressed as%
\begin{equation}
\int_{a}^{b}fdg=(-1)^{\alpha }\int_{a}^{b}D_{a+}^{\alpha }f\left(
t\right) D_{b-}^{1-\alpha }g_{b-}\left( t\right) dt\,,
\label{e.2.4}
\end{equation}%
where $g_{b-}\left( t\right) =g\left( t\right) -g\left( b\right) $.
\end{proposition}

It is known that almost all the paths of $B^H$ are $(H-\e)$-H\"older continuous for $\e\in(0,H)$. By the Fernique Theorem, we have the following estimation for its H\"older norm. 
\begin{lemma}\label{lemma-Holder continuity of fbm}
There exists a positive constant $\alpha=\alpha(H,\e, T)$ depending on $(H, \e, T)$, such that
\begin{align*}
	\bE \left[ e^{\alpha \|B^H\|_{0,T; H-\e}^2} \right] < \infty.
\end{align*}
\end{lemma}

\subsection{H\"older continuity}

In this subsection, for the solution to \eqref{SDE}, we provide some estimations for its H\"older norm, following the approach developed  in \cite[Theorem 2]{Hu2007}. 
\begin{theorem} \label{Thm-Holder1}
Suppose that the coefficient functions $\sigma$ and $b$ are bounded
and have bounded derivatives which are H\"{o}lder continuous of order
greater than $1/(H-\e)-1$ for some $\e\in(0,H-\frac12)$. Then there
exists a constant $C(H,\e)$ that depends on $(H,\e)$ only, such that
for all $T>0$ and $ 0 \le s < t \le T$,
\begin{align} \label{eq-thm1}
	|X_t - X_s| &\le C(H,\e) \|\sigma\|_{\infty} \Big[ \|B^H\|_{0,T;H-\e} (t-s)^{H-\e} \vee \|B^H\|_{0,T;H-\e}^{1/(H-\e)}  \|\sigma'\|_{\infty}^{1/(H-\e)-1} (t-s)\Big] \nonumber \\
	&\quad + 2\|b\|_{\infty} (t-s).
\end{align}
Consequently, there exists a random variable $\xi$ with $\bE|\xi|^p<\infty$ for all $p>1$ such that for $ 0 \le s < t \le T$,
\[|X_t(\omega)-X_s(\omega)|\le \xi(\omega)|t-s|^{H-\e}, a.s.\]
\end{theorem}

\begin{remark}
Note that the pathwise H\"{o}lder continuity of the process $X_t$ was established in  \cite[Theorem 2.1]{nr02} under some Lipschitz conditions  and growth conditions on the coefficient functions $b$ and $\sigma$. They also obtained the $L^p$-integrability of the H\"{o}lder norm of the process for $H>\frac12+\frac\gamma4$ where $\gamma\in[0,1]$ is the growth rate of the function $\sigma$. 

 An upper bound for $\sup_{t \in [0,T]} |X_t|$ was obtained in \cite[Theorem 2]{Hu2007}, and it turns out that the techniques introduced in \cite{Hu2007} can be used to improve the estimations for the pathwise H\"{o}lder norm. This is done in Theorem \ref{Thm-Holder1} for bounded coefficient functions, and in Theorem \ref{Thm-Holder2} for unbounded coefficient functions  (with linear growth). Compared with the result in \cite{nr02}, Theorems \ref{Thm-Holder1} and \ref{Thm-Holder2} provide sharper upper bounds for the increments of the processes, which allow to establish the $L^p$-integrability of the H\"{o}lder norm for all $H > 1/2$. In contrast,   \cite[Theorem 2.1]{nr02} obtained the $L^p$-integrability for $H > 3/4$ when the coefficient functions have linear growth.
\end{remark}

\begin{proof}[Proof of Theorem \ref{Thm-Holder1}]
Fix $\alpha \in (1-H+\e, 1/2)$.  By Proposition \ref{p1},  for $0 \le s \le t \le T$, 
\begin{align} \label{eq-complement2}
	\left| \int_s^t \sigma(X_r) \circ dB_r^H  \right|	&\le \int_s^t \left| D_{s+}^{\alpha} \sigma(X_r) D_{t-}^{1-\alpha} (B_r^H - B_t^H) \right| dr.
\end{align}
By \eqref{e-fd} and Lemma \ref{lemma-Holder continuity of fbm}, for $r \in [s,t]$, we have
\begin{align} \label{eq-complement3}
	\left| D_{t-}^{1-\alpha} (B_r^H - B_t^H) \right|
	&= \dfrac{1}{\Gamma(\alpha)} \left| \dfrac{B_r^H - B_t^H}{(t-r)^{1-\alpha}} + (1-\alpha) \int_r^t \dfrac{B_r^H - B_u^H}{(u-r)^{2-\alpha}} du \right| \nonumber \\
	&\le \dfrac{1}{\Gamma(\alpha)} \dfrac{\left| B_r^H - B_t^H \right|}{|t-r|^{1-\alpha}} + \dfrac{1-\alpha}{\Gamma(\alpha)} \int_r^t \dfrac{\left| B_r^H - B_u^H \right|}{|u-r|^{2-\alpha}} du \nonumber \\
	&\le \dfrac{\|B^H\|_{s,t;H-\e}}{\Gamma(\alpha)} |t-r|^{H-\e-1+\alpha} + \dfrac{(1-\alpha) \|B^H\|_{s,t;H-\e}}{\Gamma(\alpha)} \int_r^t (u-r)^{H-\e+\alpha-2} du \nonumber \\
	&= \left( \dfrac{1}{\Gamma(\alpha)} + \dfrac{1-\alpha}{(H-\e+\alpha-1) \Gamma(\alpha)} \right) \|B^H\|_{s,t;H-\e} |t-r|^{H-\e-1+\alpha}\, ,
\end{align}
and 
\begin{align} \label{eq-complement4}
	\left| D_{s+}^{\alpha} \sigma(X_r) \right|
	&\le \dfrac{1}{\Gamma(1-\alpha)} \dfrac{\left| \sigma(X_r) \right|}{|r-s|^{\alpha}} + \dfrac{\alpha}{\Gamma(1-\alpha)} \int_s^r \dfrac{\left| \sigma(X_r) - \sigma(X_u) \right|}{|r-u|^{\alpha+1}} du \nonumber \\
	&\le \dfrac{\|\sigma\|_{\infty}}{\Gamma(1-\alpha)} |r-s|^{-\alpha} + \dfrac{\alpha \|\sigma'\|_{\infty}}{\Gamma(1-\alpha)} \int_s^r \dfrac{\left| X_r - X_u \right|}{|r-u|^{\alpha+1}} du \nonumber \\
	&\le \dfrac{\|\sigma\|_{\infty}}{\Gamma(1-\alpha)} |r-s|^{-\alpha} + \dfrac{\alpha \|\sigma'\|_{\infty} \|X\|_{s,t;H-\e}}{\Gamma(1-\alpha)} \int_s^r \dfrac{1}{|r-u|^{\alpha+1-(H-\e)}} du \nonumber \\
	&= \dfrac{\|\sigma\|_{\infty}}{\Gamma(1-\alpha)} |r-s|^{-\alpha} + \dfrac{\alpha \|\sigma'\|_{\infty} \|X\|_{s,t;H-\e}}{\Gamma(1-\alpha) (H-\e-\alpha)} |r-s|^{H-\e-\alpha}\, ,
\end{align}
 where the H\"older norm $\|X\|_{s,t;H-\e}$ is finite a.s. by \cite[Theorem 2.1]{nr02}.

By \eqref{eq-complement2}, \eqref{eq-complement3} and \eqref{eq-complement4}, we have
\begin{align} \label{eq-3.1}
	&\left| \int_s^t \sigma(X_r) \circ dB_r^H \right| \nonumber \\
	\le& \int_s^t \left( \dfrac{\|\sigma\|_{\infty}}{\Gamma(1-\alpha)} |r-s|^{-\alpha} + \dfrac{\alpha \|\sigma'\|_{\infty} \|X\|_{s,t;H-\e}}{\Gamma(1-\alpha) (H-\e-\alpha)} |r-s|^{H-\e-\alpha} \right) \nonumber \\
	&\qquad\quad  \left( \dfrac{1}{\Gamma(\alpha)} + \dfrac{1-\alpha}{(H-\e+\alpha-1) \Gamma(\alpha)} \right) \|B^H\|_{s,t;H-\e} |t-r|^{H-\e-1+\alpha} dr \nonumber \\
	=& \left( \dfrac{1}{\Gamma(\alpha)} + \dfrac{1-\alpha}{(H-\e+\alpha-1) \Gamma(\alpha)} \right) \dfrac{\|\sigma\|_{\infty}}{\Gamma(1-\alpha)} \|B^H\|_{s,t;H-\e} \int_s^t |t-r|^{H-\e-1+\alpha} |r-s|^{-\alpha} dr \nonumber \\
	& +\left( \dfrac{1}{\Gamma(\alpha)} + \dfrac{1-\alpha}{(H-\e+\alpha-1) \Gamma(\alpha)} \right) \dfrac{\alpha \|\sigma'\|_{\infty} \|X\|_{s,t;H-\e}}{\Gamma(1-\alpha) (H-\e-\alpha)} \|B^H\|_{s,t;H-\e} \int_s^t |t-r|^{H-\e-1+\alpha} |r-s|^{H-\e-\alpha} dr \nonumber \\
	=& \left( \dfrac{1}{\Gamma(\alpha)} + \dfrac{1-\alpha}{(H-\e+\alpha-1) \Gamma(\alpha)} \right) \dfrac{\|\sigma\|_{\infty}}{\Gamma(1-\alpha)} \|B^H\|_{s,t;H-\e} (t-s)^{H-\e} \beta(H-\e+\alpha,1-\alpha) \nonumber \\
	&\quad +\left( \dfrac{1}{\Gamma(\alpha)} + \dfrac{1-\alpha}{(H-\e+\alpha-1) \Gamma(\alpha)} \right) \dfrac{\alpha \|\sigma'\|_{\infty} \|X\|_{s,t;H-\e}}{\Gamma(1-\alpha) (H-\e-\alpha)} \|B^H\|_{s,t;H-\e} (t-s)^{2H-2\e} \nonumber \\
	&\quad \qquad \times  \beta(H-\e+\alpha, H-\e-\alpha+1) \nonumber \\
	\le&  C_1(H, \e) \|B^H\|_{0,T;H-\e} \Big[ \|\sigma\|_{\infty} (t-s)^{H-\e} + \|\sigma'\|_{\infty} \|X\|_{s,t;H-\e} (t-s)^{2H-2\e} \Big],
\end{align}
where $\beta(p,q)$ is the Beta function, and $C_1(H, \e)$ is a constant depending on $(H, \e)$ only.


 Hence, by \eqref{eq-3.1}, for $0 \le s \le t \le T$,
\begin{align} \label{eq-3.2}
	\left| X_t - X_s \right|
	&\le \left| \int_s^t \sigma(X_r) \circ dB_r^H \right| + \left| \int_s^t b(X_r) dr \right| \nonumber \\
	&\le C_1(H,\e) \|B^H\|_{0,T;H-\e} \Big[ \|\sigma\|_{\infty} (t-s)^{H-\e} + \|\sigma'\|_{\infty} \|X\|_{s,t;H-\e} (t-s)^{2H-2\e} \Big] \nonumber \\
	&\quad+ \|b\|_{\infty} (t-s),
\end{align}
and therefore,
\begin{align} \label{eq-3.3}
	\|X\|_{s,t;H-\e}
	&\le C_1(H,\e) \|B^H\|_{0,T;H-\e} \Big[ \|\sigma\|_{\infty} + \|\sigma'\|_{\infty} \|X\|_{s,t;H-\e} (t-s)^{H-\e} \Big] \nonumber \\
	&\quad+ \|b\|_{\infty} (t-s)^{1-H+\e}.
\end{align}
Choose $\Delta$ such that
\begin{align*}
	\Delta^{H-\e} = \dfrac{1}{2C_1(H,\e) \|B^H\|_{0,T;H-\e} \|\sigma'\|_{\infty}}\, ,
\end{align*}
and if the denominator vanishes, we set $\Delta=\infty$.

If $\Delta \ge t-s$,  \eqref{eq-3.3} yields
\begin{align*}
	\|X\|_{s,t;H-\e}
	&\le C_1(H,\e) \|B^H\|_{0,T;H-\e} \|\sigma\|_{\infty} + \dfrac{1}{2} \|X\|_{s,t;H-\e} + \|b\|_{\infty} (t-s)^{1-H+\e},
\end{align*}
and thus,
\begin{align} \label{eq-3.4}
	\left| X_t - X_s \right|
	&\le (t-s)^{H-\e} \|X\|_{s,t;H-\e} \nonumber \\
	&\le 2C_1(H,\e) \|B^H\|_{0,T;H-\e} \|\sigma\|_{\infty} (t-s)^{H-\e} + 2\|b\|_{\infty} (t-s).
\end{align}

If $\Delta < t-s$, we divide the interval $[s,t]$ into $n = [(t-s)/\Delta] + 1$ subintervals, whose lengths are smaller than $\Delta$. Let $s = u_0 < u_1 < \cdots < u_n = t$ be the endpoints of the subintervals. Then $u_i - u_{i-1} \le \Delta$ for $1 \le i \le n$. Hence, by \eqref{eq-3.4},
\begin{align} \label{eq-3.5}
	\left| X_t - X_s \right|
	&\le \sum_{i=1}^n \left| X_{u_i} - X_{u_{i-1}} \right| \nonumber \\
	&\le 2C_1(H,\e) \|B^H\|_{0,T;H-\e} \|\sigma\|_{\infty} \sum_{i=1}^n (u_i-u_{i-1})^{H-\e} + 2\|b\|_{\infty} \sum_{i=1}^n (u_i-u_{i-1}) \nonumber \\
	&\le 2C_1(H,\e) \|B^H\|_{0,T;H-\e} \|\sigma\|_{\infty} n \Delta^{H-\e} + 2\|b\|_{\infty} (t-s) \nonumber \\
	&\le 4C_1(H,\e) \|B^H\|_{0,T;H-\e} \|\sigma\|_{\infty} (t-s) \Delta^{H-\e-1} + 2\|b\|_{\infty} (t-s) \nonumber \\
	&\le 2^{1+1/(H-\e)} C_1(H,\e)^{1/(H-\e)} \|B^H\|_{0,T;H-\e}^{1/(H-\e)} \|\sigma\|_{\infty}  \|\sigma'\|_{\infty}^{1/(H-\e)-1} (t-s) + 2\|b\|_{\infty} (t-s).
\end{align}
The desired result follows from \eqref{eq-3.4} and \eqref{eq-3.5}, and the proof is concluded.
\end{proof}

For the case where  the coefficient functions $\sigma$ or $b$ are unbounded, we have the following estimation. 

\begin{theorem} \label{Thm-Holder2}
Suppose that the coefficient functions $\sigma$ and $b$ have bounded derivatives which are H\"{o}lder continuous of order  greater than $1/(H-\e)-1$ for some $\e\in (0, H-\frac12)$. Moreover, if $\|\sigma'\|_{\infty} + \|b'\|_{\infty} > 0$, then for all $T>0$, for all $ 0 \le s < t \le T$,
\begin{align} \label{eq-3.7}
	|X_t - X_s| \le C(H,\e, \sigma, b, T)^{1+(\|B^H\|_{s,t;H-\e})^{1/(H-\e)}}  (|X_0|+1) (t-s)^{H-\e},
\end{align}
where $C(H,\e,\sigma, b, T)$ is a constant depending only on $(H, \e, \sigma, b, T)$.
\end{theorem}

\begin{proof}
 
Obviously the function $g(t)=t$ is H\"{o}lder continuous of any order $\beta \in (0,1]$ with the H\"{o}lder norm $\|g\|_{0,T;\beta} = T^{1-\beta}$. Hence, by \cite[Theorem 2 (i)]{Hu2007}, we have
\begin{align} \label{eq-3.8}
	\sup_{t \in [0,T]} |X_t|
	&\le 2^{1 + C(H,\e) T \left[ \|\sigma'\|_{\infty} + \|b'\|_{\infty} + |\sigma(0)| + |b(0)| \right]^{1/(H-\e)} \left( \|B^H\|_{0,T;H-\e} + T^{1-(H-\e)} \right)^{1/(H-\e)}} (|X_0| + 1)\notag\\
	&\le C(H,\e, b,\sigma, T)^{1+(\|B^H\|_{0,t;H-\e})^{1/(H-\e)}}(|X_0|+1).
\end{align}
In this proof,  $C(H,\e)$ and $C(H, \e, b, \sigma, T)$ are generic positive constants  depending only on $(H,\e)$ and $(H, \e, b,\sigma, T)$, respectively, and they may vary in different places.

The estimations \eqref{eq-complement2} and \eqref{eq-complement3} are still valid. Instead of \eqref{eq-complement4}, we have
\begin{align} \label{eq-complement5}
	\left| D_{s+}^{\alpha} \sigma(X_r) \right|
	&= \dfrac{1}{\Gamma(1-\alpha)} \left| \dfrac{\sigma(X_r)}{(r-s)^{\alpha}} + \alpha \int_s^r \dfrac{\sigma(X_r) - \sigma(X_u)}{(r-u)^{\alpha+1}} du \right| \nonumber \\
	&\le \dfrac{1}{\Gamma(1-\alpha)} \dfrac{\left| \sigma(X_r) - \sigma(0) \right| + |\sigma(0)|}{|r-s|^{\alpha}} + \dfrac{\alpha}{\Gamma(1-\alpha)} \int_s^r \dfrac{\left| \sigma(X_r) - \sigma(X_u) \right|}{|r-u|^{\alpha+1}} du \nonumber \\
	&\le \dfrac{1}{\Gamma(1-\alpha)} \dfrac{\|\sigma'\|_\infty |X_r| + |\sigma(0)|}{|r-s|^{\alpha}} + \dfrac{\alpha \|\sigma'\|_\infty}{\Gamma(1-\alpha)} \int_s^r \dfrac{\left| X_r - X_u \right|}{|r-u|^{\alpha+1}} du \nonumber \\
	&\le \dfrac{1}{\Gamma(1-\alpha)} \dfrac{\|\sigma'\|_\infty |X_r| + |\sigma(0)|}{|r-s|^{\alpha}} + \dfrac{\alpha \|\sigma'\|_\infty \|X\|_{s,t;H-\e}}{\Gamma(1-\alpha)} \int_s^r \dfrac{1}{|r-u|^{\alpha+1-(H-\e)}} du \nonumber \\
	&\le \dfrac{1}{\Gamma(1-\alpha)} \dfrac{\|\sigma'\|_\infty |X_r| + |\sigma(0)|}{|r-s|^{\alpha}} + \dfrac{\alpha \|\sigma'\|_\infty \|X\|_{s,t;H-\e}}{\Gamma(1-\alpha) (H-\e-\alpha)} |r-s|^{H-\e-\alpha}.
\end{align}
By \eqref{eq-complement2}, \eqref{eq-complement3}, \eqref{eq-complement5} and \eqref{eq-3.8}, we have
\begin{align} \label{eq-complement6}
	&\left| \int_s^t \sigma(X_r) \circ dB_r^H \right| \nonumber \\
	&\le \int_s^t \left( \dfrac{1}{\Gamma(\alpha)} + \dfrac{(1-\alpha)}{(H-\e+\alpha-1) \Gamma(\alpha)} \right) \|B^H\|_{s,t;H-\e} |t-r|^{H-\e-1+\alpha} \nonumber \\
	&\quad \left( \dfrac{1}{\Gamma(1-\alpha)} \dfrac{\|\sigma'\|_\infty |X_r| + |\sigma(0)|}{|r-s|^{\alpha}} + \dfrac{\alpha \|\sigma'\|_\infty \|X\|_{s,t;H-\e}}{\Gamma(1-\alpha) (H-\e-\alpha)} |r-s|^{H-\e-\alpha} \right) dr \nonumber \\
	&\le \dfrac{(H-\e) \|B^H\|_{s,t;H-\e}}{\Gamma(1-\alpha) \Gamma(\alpha) (H-\e+\alpha-1)} \left( \|\sigma'\|_\infty \sup_{r \in [0,T]} |X_r| + |\sigma(0)| \right) \int_s^t |t-r|^{H-\e-1+\alpha} |r-s|^{-\alpha} dr \nonumber \\
	&\quad + \dfrac{\alpha (H-\e) \|\sigma'\|_\infty \|X\|_{s,t;H-\e}}{\Gamma(1-\alpha) \Gamma(\alpha) (H-\e-\alpha) (H-\e+\alpha-1)} \|B^H\|_{s,t;H-\e} \int_s^t |t-r|^{H-\e-1+\alpha} |r-s|^{H-\e-\alpha} dr \nonumber \\
	&= \dfrac{(H-\e) \|B^H\|_{s,t;H-\e}}{\Gamma(1-\alpha) \Gamma(\alpha) (H-\e+\alpha-1)} \left( \|\sigma'\|_\infty \sup_{r \in [0,T]} |X_r| + |\sigma(0)| \right) |t-s|^{H-\e} \beta(H-\e + \alpha, 1-\alpha) \nonumber \\
	&\quad + \dfrac{\alpha (H-\e) \|\sigma'\|_\infty \|X\|_{s,t;H-\e}}{\Gamma(1-\alpha) \Gamma(\alpha) (H-\e-\alpha) (H-\e+\alpha-1)} \|B^H\|_{s,t;H-\e} |t-s|^{2H-2\e} \beta(H-\e + \alpha, H-\e - \alpha + 1) \nonumber \\
	&\le C(H,\e) \|B^H\|_{s,t;H-\e} |t-s|^{H-\e} \left[ \|\sigma'\|_\infty \sup_{r \in [0,T]} |X_r| + |\sigma(0)| + \|\sigma'\|_\infty \|X\|_{s,t;H-\e} |t-s|^{H-\e} \right].
\end{align}
 Similarly, it is easy to see that
\begin{align} \label{eq-complement7}
	\left| \int_s^t b(X_r) dr \right| &\le |t-s|  \left[ \|b'\|_\infty \sup_{r \in [0,T]} |X_r| + |b(0)| \right]\notag\\
	&\le  |t-s|^{H-\e} T^{1-(H-\e)}  \left[ \|b'\|_\infty \sup_{r \in [0,T]} |X_r| + |b(0)| \right].
\end{align}

 Hence, by \eqref{eq-complement6} and \eqref{eq-complement7}, there exists a constant $C_0(H,\e)$ such that 
\begin{align*}
	\left| X_t - X_s \right|
	&\le \left| \int_s^t \sigma(X_r) \circ dB_r^H \right| + \left| \int_s^t b(X_r) dr \right| \nonumber \\
	&\le C_0(H,\e) \left( \|B^H\|_{s,t;H-\e} + T^{1-(H-\e)} \right) |t-s|^{H-\e} \Big[ \left( \|\sigma'\|_\infty + \|b'\|_\infty \right) \sup_{r \in [0,T]} |X_r|  \nonumber \\
	&\quad +  |\sigma(0)| + |b(0)|  +  \|\sigma'\|_\infty \|X\|_{s,t;H-\e} |t-s|^{H-\e} \Big],
\end{align*}
and consequently, 
\begin{align} \label{eq-complement8}
	\|X\|_{s,t;H-\e}
	&\le C_0(H, \e) \left( \|B^H\|_{s,t;H-\e} + T^{1-(H-\e)} \right) \Big[ \left( \|\sigma'\|_\infty + \|b'\|_\infty \right) \sup_{r \in [0,T]} |X_r|  \nonumber \\
	&\quad + |\sigma(0)| + |b(0)| + \|\sigma'\|_\infty \|X\|_{s,t;H-\e} |t-s|^{H-\e} \Big].
\end{align}

Fix a positive constant $\Delta$ such that
\begin{align*}
	\Delta \le \left( \dfrac{1}{3C_0(H,\e) \|\sigma'\|_\infty  (\|B^H\|_{0,T;H-\e} + T^{1-(H-\e)})} \right)^{1/(H-\e)}\, . 
\end{align*}

   If $t - s \le \Delta$, we have
\begin{align} \label{eq-3.9}
	\|X\|_{s,t;H-\e}
	\le \dfrac{3}{2} C_0(H,\e) \left( \|B^H\|_{0,T;H-\e} + T^{1-(H-\e)} \right) \left( |\sigma(0)| + |b(0)| + \left( \|\sigma'\|_\infty + \|b'\|_\infty \right) \sup_{t \in [0,T]} |X_t| \right).
\end{align}
Then by \eqref{eq-3.8} and \eqref{eq-3.9},  we have 
\begin{align} \label{eq-3.10'}
 \|X\|_{s,t;H-\e} 
	\le C(H,\e, \sigma, b, T)^{1+(\|B^H\|_{s,t;H-\e})^{1/(H-\e)}} (|X_0|+1). 
\end{align}

If $t-s > \Delta$, Similar to the proof of Theorem \ref{Thm-Holder1},  we divide the interval $[s,t]$ into $n = [(t-s)/\Delta] + 1$ subintervals, whose lengths are smaller than $\Delta$. Let $s = u_0 < u_1 < \cdots < u_n = t$ be endpoints of the subintervals. Then by \eqref{eq-3.8} and \eqref{eq-3.9}, we have
\begin{align} \label{eq-3.11}
	|X_t - X_s| &\le \sum_{i=1}^n |X_{u_i} - X_{u_{i-1}}| \nonumber \\
	&\le C(H,\e, \sigma, b, T)^{1+(\|B^H\|_{s,t;H-\e})^{1/(H-\e)}} (|X_0|+1)\sum_{i=1}^n (u_i - u_{i-1})^{H-\e} \nonumber \\
	&\le C(H,\e, \sigma, b, T)^{1+(\|B^H\|_{s,t;H-\e})^{1/(H-\e)}} (|X_0|+1)n\Delta^{H-\e} \nonumber \\
	&\le C(H,\e, \sigma, b, T)^{1+(\|B^H\|_{s,t;H-\e})^{1/(H-\e)}} (|X_0|+1) (t-s) \varDelta^{H - \e - 1} \nonumber \\
	&\le C(H,\e, \sigma, b, T)^{1+(\|B^H\|_{s,t;H-\e})^{1/(H-\e)}}  (|X_0|+1) (t-s)^{H-\e}.  \end{align}
The desired result \eqref{eq-3.7} come from \eqref{eq-3.10'} and \eqref{eq-3.11}.
\end{proof}

Combining Theorem \ref{Thm-Holder1}, Theorem \ref{Thm-Holder2},  and Lemma  \ref{lemma-Holder continuity of fbm} with dominated convergence theorem, one can prove the following corollary.  
\begin{corollary} \label{Coro-upper bound}
 Assume that $\bE[|X_0|^p] < \infty$ for $p\ge 2$.  If the conditions in Theorem \ref{Thm-Holder1} or in Theorem \ref{Thm-Holder2} hold,  then $\bE[|X_t|^p]$ is a continuous function of $t$. 
\end{corollary}

\section{High-dimensional limit  for Wigner-type matrices}\label{sec:hdl}

\subsection{ Relative compactness of empirical spectral measure-valued processes}

Denote by $\mathbf P(\bR)$ the space of probability measures equipped
with weak topology, then $\mathbf P(\bR)$ is a Polish space. For  $T >
0$, let $C([0,T],\mathbf P(\bR))$ be the space of continuous $\mathbf
P(\bR)$-valued  processes.
In this subsection, under proper conditions, we obtain the {\em almost sure} relative compactness in $C([0,T], \mathbf P(\bR))$ for the set $\{L_N(t), t\in[0,T]\}_{N\in\bN}$ of empirical spectral measures   of  $Y^N(t)$ with entries given in \eqref{matrix entries Wigner}  by using Lemma \ref{lemma in Anderson}(\cite[Lemma 4.3.13]{Anderson2010}).

\begin{theorem} \label{Thm-tightness}
Assume one of the following hypotheses holds:
\begin{itemize}
	\item[H1.] Conditions in Theorem \ref{Thm-Holder1} hold;
	\item[H2.] Conditions in Theorem \ref{Thm-Holder2} hold and $\bE [|X_0|^4] < \infty$.
\end{itemize}
Assume that there exists a positive function $\varphi(x) \in C^1(\bR)$ with bounded derivative, such that $\lim_{|x| \rightarrow \infty} \varphi(x) =  \infty$ and
\begin{align*}
	\sup_{N \in \bN} \langle \varphi, L_N(0) \rangle \le C_0,
\end{align*}
for some positive constant $C_0$ almost surely. Then for any $T > 0$, the sequence $\{L_N(t), t \in [0,T]\}_{N \in \bN}$ is relatively compact in $C([0,T],\mathbf P(\bR))$ almost surely.
\end{theorem}

\begin{proof} Under hypothesis H1 (hypothesis H2, resp.), by Theorem \ref{Thm-Holder1} (Theorem \ref{Thm-Holder2}, resp.), we have 
\begin{align*}
	&\quad |X_t - X_s| \le \xi |t-s|^{H-\e}
\end{align*}
where $\xi$ is a random variable with finite moment of any order (with finite $4$-th order moment, resp.).  

 By \eqref{eq-em} and Lemma \ref{lemma-Hoffman ineq}, for any $f \in C^1(\bR)$ with bounded derivative,
\begin{align} \label{eq-4.4}
	&\quad \left| \langle f, L_N(t) \rangle - \langle f, L_N(s) \rangle \right|^2 = \left| \dfrac{1}{N} \sum_{i=1}^N \left( f(\lambda_i^N(t)) - f(\lambda_i^N(s)) \right) \right|^2 \nonumber \\
	&\le \dfrac{1}{N} \sum_{i=1}^N \left| f(\lambda_i^N(t)) - f(\lambda_i^N(s)) \right|^2 \le \dfrac{1}{N} \|f'\|_{\infty}^2 \sum_{i=1}^N \left| \lambda_i^N(t) - \lambda_i^N(s) \right|^2 \nonumber \\
	&\le \dfrac{1}{N} \|f'\|_{\infty}^2 \bT \left(Y^N(t) - Y^N(s) \right)^2 = \dfrac{1}{N} \|f'\|_{\infty}^2 \sum_{i,j=1}^N \left( Y_{ij}^N(t) - Y_{ij}^N(s) \right)^2 \nonumber \\
	&= \|f'\|_{\infty}^2 \left[ \dfrac{2}{N^2} \sum_{i<j}\left( X_{ij}^N(t) - X_{ij}^N(s) \right)^2 + \dfrac{2}{N^2} \sum_{i=1}^N \left( X_{ii}^N(t) - X_{ii}^N(s) \right)^2 \right] \nonumber \\
	&\le \|f'\|_{\infty}^2 (t-s)^{2H-2\e} \left( \dfrac{2}{N^2} \sum_{i<j} \xi_{ij}^2 + \dfrac{2}{N^2} \sum_{i=1}^N \xi_{ii}^2 \right) \nonumber \\
	&\le  \frac{4}{N(N+1)}\|f'\|_{\infty}^2 (t-s)^{2H-2\e}  \sum_{i\le j} \xi_{ij}^2 ,
\end{align}
where $\{\xi_{ij}\}$ are i.i.d copies of $\xi$. 

Recall that by the Arzela-Ascoli Theorem  (or Lemma \ref{Lemma-B4}), for any number $M > 0$, the set
\begin{align*}
	 \bigcap_{n=1}^{\infty} \left\{ g \in C([0,T], \bR): \sup_{s,t \in [0,T], |t-s| \le \eta_n} |g(t) - g(s)| \le \e_n, \quad \sup_{t \in [0,T]} |g(t)| \le M \right\},
\end{align*}
where $\{\e_n, n \in \bN\}$ and $\{\eta_n, n \in \bN\}$ are two sequences of positive real numbers going to zero as $n$ goes to infinity, is compact in $C([0,T], \bR)$.

Define
\begin{align*}
	K = \left\{ \nu \in \mathbf P(\bR): \int \varphi(x) \nu(dx) \le C_0 + M_0 \right\},
\end{align*}
where $M_0$ is a positive number that will be determined later. Then  by Lemma \ref{Lemma-B3}, $K$ is compact in $\mathbf P(\bR)$. Note that there exists a sequence of $C_b^1(\bR)$ functions $\{f_i\}_{i \in \bN}$ that it is dense in $C_0(\bR)$. Choose a positive integer $p_0$, such that $p_0(H-\e) > 1$, and define
\begin{align*}
	C_T(f_i) = \bigcap_{n=1}^{\infty} \left\{ \nu \in C([0,T], \mathbf P(\bR)): \sup_{s,t \in [0,T], |t-s| \le n^{-p_0}} |\langle f_i, \nu_t \rangle - \langle f_i, \nu_s \rangle| \le M_i n^{1 - p_0(H-\e)} \right\},
\end{align*}
where $\{M_i\}_{i \in \bN}$ is a sequence of positive numbers that are independent of $n$ and will be determined later. Denote
\begin{align*}
	\mathfrak{C} = \left\{ \nu \in C([0,T], \mathbf P(\bR)): \nu_t \in K, \forall t \in [0,T] \right\} \cap \bigcap_{i=1}^{\infty} C_T(f_i),
\end{align*}
then $\mathfrak{C}$ is compact in $C([0,T], \mathbf P(\bR))$, according to Lemma \ref{lemma in Anderson} (\cite[Lemma 4.3.13]{Anderson2010}). Hence, it is enough to show \begin{align*}
	\bP \left( \liminf_{N \rightarrow \infty} \big\{ L_N \in \mathfrak{C} \big\} \right) = 1.
\end{align*}

Note that $\xi_{ij}^2 - \bE[\xi_{ij}^2]:=\xi_{ij}^2-m$ are i.i.d. with mean zero and finite variance denoted by $\sigma^2$ for $1\le i\le j\le N$. By \eqref{eq-4.4} and the Markov inequality, we have
\begin{align} \label{eq-4.4'}
	&\quad \bP \left(  L_N \notin \Big\{ \nu \in C([0,T], \mathbf P(\bR)): \nu_t \in K, \forall t \in [0,T] \Big\} \right) \nonumber \\
	&= \bP \Big( \exists t \in [0,T], L_N(t) \notin K \Big) = \bP \left( \sup_{t \in [0,T]} \langle \varphi, L_N(t) \rangle > C_0 + M_0 \right) \nonumber \\
	&\le \bP \left( \sup_{t \in [0,T]} \left| \langle \varphi, L_N(t) \rangle - \langle \varphi, L_N(0) \rangle \right| > M_0 \right) \nonumber \\
	&\le \bP \left( \frac{4}{N(N+1)}\|\varphi'\|_{\infty}^2 T^{2H-2\e}  \sum_{i\le j} \xi_{ij}^2  > M_0^2 \right) \nonumber \\
	&= \bP \left( \dfrac{2}{N(N+1)} \sum_{i\le j} \xi_{ij}^2 - m > \dfrac{M_0^2}{2\|\varphi'\|_{\infty}^2 T^{2H-2\e}} - m \right) \nonumber \\
	&\le \left( \dfrac{M_0^2}{2\|\varphi'\|_{\infty}^2 T^{2H-2\e}} - m \right)^{-2} \bE \left[ \left( \dfrac{2}{N(N+1)} \sum_{i\le j} \xi_{ij}^2 - m \right)^2 \right] \nonumber \\
	&= \left( \dfrac{M_0^2}{2\|\varphi'\|_{\infty}^2 T^{2H-2\e}} - m \right)^{-2} \dfrac{4}{N^2(N+1)^2} \bE \left[ \left( \sum_{i\le j}\big[\xi_{ij}^2 - m\big] \right)^2 \right] \nonumber \\
	&= \left( \dfrac{M_0^2}{2\|\varphi'\|_{\infty}^2 T^{2H-2\e}} - m \right)^{-2} \dfrac{4}{N^2(N+1)^2} \bE \left[ \sum_{i\le j} \big(\xi_{ij}^2 - m\big)^2 \right] \nonumber \\
	&= \left( \dfrac{M_0^2}{2\|\varphi'\|_{\infty}^2 T^{2H-2\e}} - m \right)^{-2} \dfrac{2}{N(N+1)} \sigma^2.
\end{align}
If we choose $M_0 = 2\|\varphi'\|_{\infty} T^{H-\e} (1+m)$, then \eqref{eq-4.4'} becomes
\begin{align} \label{eq-4.5}
	&\quad \bP \left( L_N \notin \Big\{ \nu \in C([0,T], \mathbf P(\bR)): \nu_t \in K, \forall t \in [0,T] \Big\} \right) \nonumber \\
	&\le \dfrac{2\sigma^2}{N(N+1) (2m^2 + 3m + 2)^2} \le \dfrac{8\sigma^2}{N(N+1)} .
\end{align}
Similarly, by \eqref{eq-4.4} and the Markov inequality, we have
\begin{align} \label{eq-4.6}
	&\quad \bP \left(  L_N \notin C_T(f_i) \right) \nonumber \\
	&\le \sum_{n=1}^{\infty} \bP \left( L_N \notin \left\{ \nu \in C([0,T], \mathbf P(\bR)): \sup_{s,t \in [0,T], |t-s| \le n^{-p_0}} |\langle f_i, \nu_t \rangle - \langle f_i, \nu_s \rangle| \le M_i n^{1 - p_0(H-\e)} \right\} \right) \nonumber \\
	&= \sum_{n=1}^{\infty} \bP \left( \sup_{s,t \in [0,T], |t-s| \le n^{-p_0}} |\langle f_i, L_N(t) \rangle - \langle f_i, L_N(s) \rangle| > M_i n^{1 - p_0(H-\e)} \right) \nonumber \\
	&\le \sum_{n=1}^{\infty} \bP \left( 2\|f_i'\|_{\infty}^2 n^{-p_0(2H-2\e)} \left( \dfrac{2}{N(N+1)} \sum_{i\le j} \xi_{ij}^2 \right) > M_i^2 n^{2 - 2p_0(H-\e)} \right) \nonumber \\
	&= \sum_{n=1}^{\infty} \bP \left( \dfrac{2}{N(N+1)} \sum_{i\le j} \xi_{ij}^2 - m > \dfrac{M_i^2 n^2}{2\|f_i'\|_{\infty}^2} - m \right) \nonumber \\
	&\le \sum_{n=1}^{\infty} \left( \dfrac{M_i^2 n^2}{2\|f_i'\|_{\infty}^2} - m \right)^{-2} \bE \left[ \left( \dfrac{2}{N(N+1)} \sum_{i\le j} \xi_{ij}^2 - m \right)^2 \right] \nonumber \\
	&= \sum_{n=1}^{\infty} \left( \dfrac{M_i^2 n^2}{2\|f_i'\|_{\infty}^2} - m \right)^{-2} \dfrac{4}{N^2(N+1)^2} \bE \left[ \sum_{i\le j} (\xi_{ij}^2 - m)^2 \right] \nonumber \\
	&\le \sum_{n=1}^{\infty} \left( \dfrac{M_i^2 n^2}{2\|f_i'\|_{\infty}^2} - m \right)^{-2} \dfrac{2}{N(N+1)} \sigma^2.
\end{align}
Now, we choose $M_i = 2 i \|f_i'\|_{\infty} \gamma$, where $\gamma$ is a positive real number such that $2\gamma^2>m$. Then \eqref{eq-4.6} becomes
\begin{align} \label{eq-4.7'}
	\bP \left( L_N \notin C_T(f_i) \right)
	&\le \sum_{n=1}^{\infty} \dfrac{2\sigma^2}{N(N+1)(2i^2n^2{\gamma}^2 - m)^2}.
\end{align}
Hence, by the definition of $\mathfrak{C}$, \eqref{eq-4.5} and \eqref{eq-4.7'}, we have
\begin{align} \label{eq-4.8}
	\sum_{N=1}^{\infty} \bP \left( L_N \notin \mathfrak{C} \right)
	&\le \sum_{N=1}^{\infty} \bP \left( L_N \notin \left\{ \nu \in C([0,T], \mathbf P(\bR)): \nu_t \in K, \forall t \in [0,T] \right\} \right) \nonumber \\
	&\qquad + \sum_{N=1}^{\infty} \sum_{i=1}^{\infty} \bP \left( L_N(t) \notin C_T(f_i) \right) \nonumber \\
	&\le \sum_{N=1}^{\infty} \dfrac{8\sigma_D^2}{N(N+1)} + \sum_{N=1}^{\infty} \sum_{i=1}^{\infty} \sum_{n=1}^{\infty} \dfrac{2\sigma^2}{N(N+1)(2i^2n^2\gamma^2 - m)^2} \nonumber \\
	&<\infty.
\end{align}
Therefore, by the Borel-Cantelli Lemma and \eqref{eq-4.8}, we have
\begin{align*}
	\bP \left( \limsup_{N \rightarrow \infty} \left\{ L_N \notin \mathfrak{C} \right\} \right) = 0.
\end{align*}
The proof is concluded. 
\end{proof}

\subsection{Limit of  empirical  spectral distributions}

Recall that  the celebrated semicircle  distribution $\mu_{sc}(dx)$ on $[-2,2]$ has density function
\begin{align}\label{e:semi-circ}
	p_{sc}(x) = \dfrac{\sqrt{4-x^2}}{2\pi} 1_{[-2,2]}(x).
\end{align}

\begin{theorem} \label{Thm-LED}
Suppose that the conditions in Theorem \ref{Thm-tightness} hold. We
also assume that $\bE [|X_0|^2] < \infty$ and denote $m_t = \bE [X_t]$
and $d_t = (\bE[|X_t|^2] - m_t^2)^{1/2}$. Then for any $T > 0$, the
sequence  $\{L_N(t), t\in[0,T]\}_{N \in \bN}$
converges to $\mu = \{\mu_t, t\in[0,T]\}$ in $C([0,T],\mathbf P(\bR))$
almost surely, as $N\to \infty$. The limit measure $\mu_t$ has density
$p_t(x) = p_{sc}(x/d_t)/d_t$, where $p_{sc}(x)$ is given by \eqref{e:semi-circ}.
\end{theorem}

\begin{proof}
First, for fixed $t\in[0,T]$, we prove the almost sure weak convergence of the empirical measure $L_N(t)$. 

Note that by Corollary \ref{Coro-upper bound}, $m_t$ and $d_t$ are continuous functions of $t$ on  $[0,T]$.  Let $\widetilde{Y}^N(t) = \left( \widetilde{Y}_{ij}^N(t) \right)_{1 \le i, j \le N}$ be a symmetric matrix with entries
\begin{align}\label{eq:Y1}
	\widetilde{Y}_{ij}^N(t) =
	\dfrac{Y_{ij}^N(t) - m_t/\sqrt{N}}{d_t}, ~~ 1 \le i \le j \le N.
\end{align}
Then $\widetilde{Y}^N(t)$ is a Wigner matrix. Let $\tilde{\lambda}_1^N(t) \le \cdots \le \tilde{\lambda}_N^N(t)$ be the eigenvalues of $\widetilde{Y}^N(t)$ and $\widetilde{L}_N(t)(dx) = \sum_{i=1}^N \delta_{\tilde{\lambda}_i^N(t)} (dx) / N$ be the empirical spectral measure. By Lemma \ref{lemma-Wigner semicircle},
$\widetilde{L}_N(t)(dx)$ converges weakly to  the semicircle distribution $p_{sc}(x) dx$ almost surely for all $0 \le t \le T$. Hence  the empirical measure of the eigenvalues of $d_t \widetilde{Y}^N(t)$ converges weakly to $\frac1{d_t}p_{sc}\left(\frac{x}{d_t}\right)dx$ almost surely.  

Note that by  \eqref{eq:Y1}, $Y^N(t) = d_t \widetilde{Y}^N(t) + \dfrac{m_t}{\sqrt{N}} E_N$, where $E_N$ is an $N \times N$ matrix with unit entries. Then by \cite[Exercise 2.4.4]{Tao2012}, we can conclude that the empirical distribution
$	L_N(t)(dx)$  converges weakly to $\dfrac{1}{d_t} p_{sc}\left(\dfrac{x}{d_t}\right) dx$, almost surely, for all $t \in [0,T]$.

Now, we show that $\{L_N(t), t \in [0,T]\}_{N \in \bN}$ converges weakly to $ \{\mu_t, t \in [0,T]\}_{N \in \bN}$ almost surely. 

Let $\{L_{N_i}(t), t \in [0,T]\}_{N \in \bN}$ be an arbitrary convergent subsequence with the limit $\{\mu_t, t\in[0,T]\}$, then for $t \in [0,T]$, $\mu_t(dx)= p_{sc}(x/d_t)/d_t dx$.  Therefore, noting that Theorem \ref{Thm-tightness} yields that the sequence $\{L_N(t), t \in [0,T]\}_{N \in \bN}$ is relatively compact almost surely in $C([0,T],\mathbf P(\bR))$, any  subsequence of  $\{L_N(t), t \in [0,T]\}_{N \in \bN}$ has a convergent subsequence with the unique limit $\{\mu_t(dx) = p_{sc}(x/d_t)/d_t dx, t\in[0,T]\}$. 
This  implies that the total sequence $\{L_N(t), t \in [0,T]\}_{N \in \bN}$ converges in $C([0,T],\mathbf P(\bR))$ with the limit $\{\mu_t(dx) = p_{sc}(x/d_t)/d_t dx, t\in[0,T]\}$, almost surely. 
\end{proof}

\begin{remark} \label{remark-6}
If $\sigma(x) =1$, $b(x) = 0$ and $X_0 = 0$, then the solution to the SDE \eqref{SDE} is the fractional Brownian motion $X_t = B_t^H$, and Theorem \ref{Thm-LED} implies that the the empirical spectral measure converges weakly to the scaled semicircle distribution with density $p_t(x) = p_{sc}(x/t^H)/t^H$ {\em almost surely}. This improves the results in \cite{Pardo2016}, where the convergence of the empirical spectral measure {\em in law} is obtained.
\end{remark}

\begin{remark} \label{remark-1'}
The Stieltjes transform $G_t(z)$ of the limiting measure $\mu_t$ is
\begin{align*}
	G_t(z) = \int \dfrac{\mu_t(dx)}{z-x}
	= \int \dfrac{p_{sc}(x/{d_t})/d_t}{z-x} dx
	= \int \dfrac{p_{sc}(x)}{z-d_t x} dx
	= \dfrac{1}{d_t} G_{sc}(z/d_t),
\end{align*}
where
\begin{align*}
	G_{sc}(z) = \int \dfrac{p_{sc}(x)}{z-x} dx,
\end{align*}
is the Stieltjes transform of the semi-circle distribution. If we assume that the variance $d_t^2$ of the solution $X_t$ is continuously differentiable on $(0,T)$,  then we have
\begin{align}
	\partial_t G_t(z)
	&= 	\partial_t  \int \dfrac{p_{sc}(x)}{z-d_t x} dx
	= d_t' \int \dfrac{xp_{sc}(x)}{(z-d_t x)^2} dx\notag \\
	&= \dfrac{d_t'}{d_t} \int \left( -\dfrac{1}{z - d_t x} + \dfrac{z}{(z - d_t x)^2} \right) p_{sc}(x) dx \notag\\
	&= - \dfrac{d_t'}{d_t} G_t(z) - \dfrac{d_t'}{d_t} z \partial_z G_t(z).\label{eq:G-t}
\end{align}
By \cite[Lemma 2.11]{Bai2010}, it is easy to get
\begin{align*}
	d_t^2 G_t(z)^2 = z G_t(z) - 1,
\end{align*}
and by \eqref{eq:G-t},
\begin{align} \label{equation for Stieltjes-SC}
	\partial_t  G_t(z)
	&= - \dfrac{d_t'}{d_t} (G_t(z) + z \partial_z G_t(z))=- \dfrac{d_t'}{d_t} \partial_z(zG_t(z)-1)\notag \\
	&= - d_t d_t' \partial_z (G_t(z)^2)
	= - (d_t^2)' G_t(z) \partial_z G_t(z).
\end{align}
For the case $X_t = B_t^H$ with $d_t^2 = t^{2H}$, equation \eqref{equation for Stieltjes-SC} becomes
\begin{align} \label{equation for Stieltjes-SC-fbm}
	\partial_t  G_t(z)
	= - 2H t^{2H-1} G_t(z) \partial_z G_t(z).
\end{align}
Denoting $F_t(z)=G_{t^{1/2H}}(z)$, by change of variable and \eqref{equation for Stieltjes-SC-fbm}, one can deduce that $F_t(z)$ satisfies the complex Burgers' equation
\begin{align*}
	\partial_t F_t(z) = -F_t(z) \partial_z F_t(z).
\end{align*}
This relationship was obtained in \cite{Jaramillo2019}.
\end{remark}

\subsection{Complex case} \label{subsec:complex}

In this subsection, we consider the following 2-dimensional SDE for $Z_t=(Z_t^{(1)}, Z_t^{(2)})$,
\begin{align} \label{SDE-complex}
	dZ_t = \tilde{\sigma} (Z_t) \circ dB_t^H + \tilde{b}(Z_t) dt,
\end{align}
with initial value $Z_0$ that is independent of $B_t^H$. Here $\tilde{b} : \bR^2 \rightarrow \bR^2$, $\tilde{\sigma} : \bR^2 \rightarrow \bR^{2\times 2}$ are continuously differentiable functions  and $B_t^H$ is a 2-dimensional fractional Brownian motion. By \cite{Lyons1994},  there exists a unique solution to SDE \eqref{SDE-complex}, if  $\tilde{\sigma}$ and $\tilde{b}$ have bounded derivatives which are H\"{o}lder continuous of order greater than $1/H-1$. 

Denote by $\iota$ the imaginary unit.  Let $\{Z_{kl}\}_{k,l \ge 1}$ be i.i.d. copies of $Z_t^{(1)} + \iota Z_t^{(2)}$ and $Z^N(t) = \left( Z_{kl}^N(t) \right)_{1\le k,l \le N}$ be a Hermitian $N \times N$ matrix with entries
\begin{align} \label{matrix entries Wigner-complex}
	Z_{kl}^N(t) =
	\begin{cases}
		\dfrac{1}{\sqrt{N}} Z_{kl}(t), & 1 \le k < l \le N, \\
		\dfrac{1}{\sqrt{N}} X_{kk}(t), & 1 \le k \le N.
	\end{cases}
\end{align}
Here, $\{X_{kk}(t)\}$ are i.i.d. copies of the real-valued process $X_t$ satisfying \eqref{SDE} and independent of the family $\{Z_{kl}(t)\}_{k,l \ge 1}$. Let $\lambda_1^N(t) \le \cdots \le \lambda_N^N(t)$ be the eigenvalues of $Z^N(t)$ and denote the empirical spectral measure by 
\begin{align*}
	L_N(t)(dx) = \dfrac{1}{N} \sum_{k=1}^N \delta_{\lambda_k^N(t)}(dx).
\end{align*}

\begin{theorem} \label{Thm-LED complex}
Suppose that the coefficient functions $\tilde{\sigma}$, $\tilde{b}$, $\sigma$ and $b$ have bounded (partial) derivatives which are H\"{o}lder continuous of order greater than $1/(H-\e)-1$. Besides, assume that among the following conditions,
\begin{itemize}
	\item [$(a_1)$] $\tilde{\sigma}$ and $\tilde{b}$ are bounded and $\bE[\|Z_0\|^2] < \infty$;
	\item [$(a_2)$] $\|(\tilde{\sigma}_x, \tilde \sigma_y)\|_{L^{\infty}(\bR^2)} + \|(\tilde{b}_x, \tilde b_y)\|_{L^{\infty}(\bR^2)} > 0$, $\bE[\|Z_0\|^4] < \infty$;
	\item [$(b_1)$] $\sigma$ and $b$ are bounded and $\bE[|X_0|^2] < \infty$;
	\item [$(b_2)$] $\|\sigma'\|_{\infty} + \|b'\|_{\infty} > 0$, $\bE[|X_0|^4] < \infty$;
\end{itemize}
$(a_1)$ or $(a_2)$ holds and $(b_1)$ or $(b_2)$ holds. Furthermore, suppose that there exists a positive function $\varphi(x) \in C^1(\bR)$ with bounded derivative, such that $\lim_{|x| \rightarrow \infty} \varphi(x) = + \infty$ and
\begin{align*}
	\sup_{N \in \bN} \langle \varphi, L_N(0) \rangle \le C_0,
\end{align*}
for some positive constant $C_0$ almost surely.
	
Then for any $T > 0$, $\bE [|X_t|^2 + \|Z_t\|^2] < \infty$ for $t \in [0,T]$, and the sequence  $\{L_N(t), t\in[0,t]\}_{N \in \bN}$ converges to $\mu = \{\mu_t, t\in[0,T]\}$ in $C([0,T],\mathbf P(\bR))$ almost surely. The limiting measure $\mu_t$ has density $p_t(x) = p_{sc}(x/d_Z(t))/d_Z(t)$, where $d_Z^2(t)$ is the variance of the solution $Z_t$ to SDE \eqref{SDE-complex}. 
\end{theorem}

\begin{proof}
The proof is similar to the real case, which is sketched below.
	
First of all, following the proof of Theorem \ref{Thm-Holder1} or Theorem \ref{Thm-Holder2}, we can still establish the pathwise H\"{o}lder continuity for the solution $Z_t = (Z_t^{(1)}, Z_t^{(2)})$ to the SDE \eqref{SDE-complex}.  More specifically, for the case that condition $(a_1)$ holds, we can obtain
\begin{align*}
	&\quad \left| Z_t^{(1)} - Z_s^{(1)} \right| + \left| Z_t^{(2)} - Z_s^{(2)} \right| \nonumber \\
	&\le C(H,\e) \|\tilde{\sigma}\|_{\infty} \left[ \|B^H\|_{0,T;H-\e} (t-s)^{H-\e} \vee \|B^H\|_{0,T;H-\e}^{1/(H-\e)} (t-s) \|\tilde{\sigma}'\|_{\infty}^{1/(H-\e)-1} \right] \nonumber \\
	&\quad + 2\|\tilde{b}\|_{\infty} (t-s),
\end{align*}
and for the case of $(a_2)$, 
\begin{align*}
	&\quad \left| Z_t^{(1)} - Z_s^{(1)} \right| + \left| Z_t^{(2)} - Z_s^{(2)} \right| \nonumber \\
	&\le C(H,\e, \tilde{\sigma}, \tilde{b}, T)^{1+(\|B^H\|_{s,t;H-\e})^{1/(H-\e)}}  (||Z_0||+1) (t-s)^{H-\e},
\end{align*}
for all $T > 0$, for all $0 \le s < t \le T$. Thus, by Lemma \ref{lemma-Holder continuity of fbm} and the moment assumption on $Z_0$, we have that for both cases, there exists a positive random variable $\zeta$ with finite second moment, such that
\begin{align} \label{eq-7.3}
	\left| Z_t^{(1)} - Z_s^{(1)} \right|^2 + \left| Z_t^{(2)} - Z_s^{(2)} \right|^2
	\le \zeta (t-s)^{2H-2\e}, ~~\text{a.s.}
\end{align}

Similar to \eqref{eq-4.4}, we have
\begin{align*}
	\quad \left| \langle f, L_N(t) \rangle - \langle f, L_N(s) \rangle \right|^2
	\le \dfrac{1}{N^2} \|f'\|_{\infty}^2 (t-s)^{2H-2\e} \left( \sum_{k=1}^N \xi_{kk}^2 + 2 \sum_{k<l} \zeta_{kl}^2 \right),
\end{align*}
for any $f \in C^1(\bR)$ with bounded derivative.  Then following the same approach used in the proof of Theorem \ref{Thm-tightness}, we can show that the sequence of empirical spectral measures $\{L_N(t), t \in [0,T]\}_{N \in \bN}$ is relatively compact in $C([0,T],\mathbf P(\bR))$ almost surely.
	
Following the proof of Corollary \ref{Coro-upper bound}, it is easy to see that the mean $m_Z(t) = \bE[Z_t]$  and the variance $d_Z^2(t)  =\bE[\|Z_t-m_Z(t)\|^2]$ are continuous functions of $t$. 
	
Next, we introduce a Hermitian matrix $\widetilde{Z}^N(t) = \left( \widetilde{Z}_{kl}^N(t) \right)_{1 \le k, l \le N}$ satisfying
\begin{align*}
	Z^N(t) = \dfrac{m_Z(t)}{\sqrt{N}} \left( E_N - I_N \right) + \frac{m_t}{\sqrt N}I_N+ d_Z(t) \widetilde{Z}^N(t).
\end{align*}
Hence,  $\widetilde{Z}^N(t)$ is a Hermitian Wigner  matrix for all $t \in [0,T]$.  Finally, by  \cite[Exercise 2.4.3, Exercise 2.4.4]{Tao2012}, we can conclude that the almost-sure limit of the empirical distribution of the eigenvalues of $Z^N(t)$ coincides with that of $d_Z \widetilde{Z}^N(t)$ for all $t \in [0,T]$.  Therefore,  by Lemma \ref{lemma-Wigner semicircle-complex}, $\{L_N(t), t \in [0,T]\}_{N \in \bN}$ converges towards $\{\mu_t, t\in[0,T]\}$ in $C([0,T], \mathbf P(\bR))$ with density $\mu_t(dx)=p_t(x)dx = p_{sc}(x/d_Z(t))/d_Z(t)dx$.
\end{proof}

\begin{remark} \label{reamrk-3'}

Similar to Remark \ref{remark-1'}, the Stieltjes transform of the limiting measure $\mu_t$ satisfies the differential equation \eqref{equation for Stieltjes-SC} with $d_t$ replaced by $d_Z(t)$.
\end{remark}

\section{High-dimensional limit  for symmetric matrices with dependent entries} \label{sec:dependent}

Let $\{X_{(i,j)}(t)\}_{i,j \in \bZ}$ be i.i.d. copies of $X_t$, the solution of  \eqref{SDE}.  Fix a finite index set $I \subset \bZ^2$ and a family of constants $\{a_r: r \in I\}$. Let $|I| = \max
\{|i_1-i_2| \vee |j_1-j_2|: (i_1,j_1), (i_2,j_2) \in I\}$  and $\#I$ be the
cardinality of  $I$. Note that $\#I \le (2|I|+1)^2$. Let $R^N(t) = \left( R_{ij}^N(t) \right)_{1 \le i,j \le N}$ be an $N \times N$ real symmetric matrix with entries
\begin{align} \label{matrix entries dependent}
	R_{ij}^N(t) = \dfrac{1}{\sqrt{N}} \sum_{r \in I} a_r X_{(i,j) + r}(t), \quad 1 \le i \le j \le N.
\end{align}
Let $\lambda_1^N(t) \le \cdots \le \lambda_N^N(t)$ be the eigenvalues of $R^N(t)$, and
\begin{align*}
	L_N(t)(dx) = \dfrac{1}{N} \sum_{i=1}^N \delta_{\lambda_i^N(t)}(dx)
\end{align*}
be the empirical spectral measure of $R^N(t)$.

\begin{theorem} \label{Thm-LED dependence}
Suppose that the conditions in Theorem \ref{Thm-tightness} hold. Then for any $T > 0$, the sequence  $\{L_N(t), t\in[0,T]\}_{N \in \bN}$ converges to $ \{\mu_t, t\in[0,T]\}$ in $C([0,T],\mathbf P(\bR))$ almost surely. The Stieltjes transform $S_t(z)=\int (z-x)^{-1}\mu_t(dx)$ of the limit measure is given by, for $z\in \bC\backslash \bR$,
\begin{align*}
	S_t(z) = \int_0^1 h_t(x,z) dx,
\end{align*}
where $h_t(x,z)$ is the solution to the equation
\begin{align*}
	h_t(x,z) = \left( -z + \int_0^1 f_t(x,y) h_t(y,z) dy \right)^{-1},
\end{align*}
with
\begin{align*}
	\quad f_t(x,y) = \sum_{k,l\in \bZ} \gamma_t(k,l) e^{-2\pi i(kx+ly)},
\end{align*}
where $\gamma_t(k,l) = \gamma_t(l,k) = d_t^2 \sum_{r \in I \cap (I + (k,l))} a_r a_{r-(k,l)}$ for $k \le l$.
\end{theorem}

\begin{proof}
 By Theorem \ref{Thm-Holder1} and Theorem \ref{Thm-Holder2}, we have the estimation
\begin{align*}
	|X_{(i,j)}(t) - X_{(i,j)}(s)| \le \xi_{(i,j)} |t-s|^{H-\e}, \forall i,j \in \bZ,
\end{align*}

where $\{\xi_{(i,j)}\}_{i,j \in \bZ}$ are i.i.d. copies of $\xi$ with $\bE|\xi|^4<\infty$. Thus,  for $1 \le i \le j \le N$,
\begin{align*}
	\left| R_{ij}^N(t) - R_{ij}^N(s) \right|
	&\le \dfrac{1}{\sqrt{N}} \sum_{r \in I} |a_r| \left| X_{(i,j) + r}(t) - X_{(i,j) + r}(s) \right| \\
	&\le \dfrac{1}{\sqrt{N}} \sum_{r \in I} |a_r| \xi_{(i,j) + r} (t-s)^{H-\e}.
\end{align*}
 Define
\begin{align*}
	F_{ij} = \left( \sum_{r \in I} |a_r| \xi_{(i,j) + r} \right)^2, 1 \le i \le j.
\end{align*}
Then all $F_{ij}$'s for $1 \le i \le j$ are distributed identically with finite second moment.

Analogous to \eqref{eq-4.4}, we have
\begin{align*}
	\left| \langle f, L_N(t) \rangle - \langle f, L_N(s) \rangle \right|^2
	&= \dfrac{1}{N} \|f'\|_{\infty}^2 \sum_{i,j=1}^N \left( R_{ij}^N(t) - R_{ij}^N(s) \right)^2 \nonumber \\
	&\le \dfrac{2}{N} \|f'\|_{\infty}^2 \sum_{i \le j} \left( R_{ij}^N(t) - R_{ij}^N(s) \right)^2 \nonumber \\
	&\le \dfrac{2}{N^2} \|f'\|_{\infty}^2 \sum_{i \le j} F_{ij} (t-s)^{2H-2\e}.
\end{align*}
Noting that the mutual independence among $\{\xi_{ij}\}$ implies the independence between $F_{ij}$ and $F_{kl}$ if $k \notin [i-|I|, i+|I|]$ or $l \notin [j-|I|, j+|I|]$, we have
\begin{align*}
	&\quad \bE \left[ \left( \sum_{i\le j} \left( F_{ij} - \bE \left[ F_{ij} \right] \right) \right)^2 \right] = \sum_{i \le j} \sum_{k \le l} \bE \left[ \left( F_{ij} - \bE \left[ F_{ij} \right] \right) \left( F_{kl} - \bE \left[ F_{kl} \right] \right) \right] \\
	&= \sum_{i \le j} \sum_{k=i-|I|}^{i+|I|} \sum_{l=j-|I|}^{j+|I|} \bE \left[ \left( F_{ij} - \bE \left[ F_{ij} \right] \right) \left( F_{kl} - \bE \left[ F_{kl} \right] \right) \right] \\
	&\le \sum_{i \le j} \sum_{k=i-|I|}^{i+|I|} \sum_{l=j-|I|}^{j+|I|} \left( \bE \left[ \left( F_{ij} - \bE \left[ F_{ij} \right] \right)^2 \right] \bE \left[ \left( F_{kl} - \bE \left[ F_{kl} \right] \right)^2 \right] \right)^{1/2} \\
	&= \dfrac{(2|I|+1)^2N(N+1)}{2} \left( \bE[F_{00}^2] -(E[F_{00}])^2 \right).
\end{align*}
   Thus, following the proof of Theorem \ref{Thm-tightness},  we may get estimations analogous to \eqref{eq-4.4'} and \eqref{eq-4.6} therein,  and then obtain the almost sure relatively compactness of the empirical spectral measure $\{L_N(t), t\in [0,T]\}_{N \in \bN}$. 

Now, let $\widetilde{R}^N(t) = \left( \widetilde{R}_{ij}^N(t) \right)_{1 \le i, j \le N}$ be a symmetric matrix with entries
\begin{align*}
	\widetilde{R}_{ij}^N(t) := R_{ij}^N(t)  -\bE[R_{ij}^N(t)]= \dfrac{1}{\sqrt{N}} \sum_{r \in I} a_r \left( X_{(i,j) + r}(t) - \bE \left[ X_{(i,j) + r}(t) \right] \right), \quad 1 \le i \le j \le N.
\end{align*}
Let $\widetilde{L}_N(t)$ be the empirical spectral measure of $\widetilde{R}^N(t)$. Then by Lemma~\ref{lemma-correlated entries}
 (\cite[Theorem 3]{Banna2015}), for each $t \in [0,T]$, $\widetilde{L}_N(t)$ converges to a deterministic probability measure $\mu_t$ almost surely. Moreover, the Stieltjes transform of the limit measure $\mu_t$ is given by 
\begin{align*}
	S_t(z) = \int_0^1 h_t(x,z) dx,
\end{align*}
where $h_t(x,z)$ is the solution to the equation
\begin{align*}
	h(x,z) = \left( -z + \int_0^1 f(x,y) h(y,z) dy \right)^{-1},
\end{align*}
with
\begin{align*}
	\quad f(x,y) = \sum_{k,l\in \bZ} \gamma_{k,l} e^{-2\pi i(kx+ly)},
\end{align*}
where,  for $k \le l$,
\begin{align*}
	\gamma_{k,l} = \gamma_{l,k} &= \bE \left[ \sum_{r \in I} a_r \left( X_r(t) - \bE \left[ X_r(t) \right] \right) \sum_{r' \in I} a_{r'} \left( X_{(k,l) + r'}(t) - \bE \left[ X_{(k,l) + r'}(t) \right] \right) \right] \\
	&= \bE \left[ \sum_{r \in I} a_r \left( X_r(t) - \bE \left[ X_r(t) \right] \right) \sum_{r' \in I + (k,l)} a_{r'-(k,l)} \left( X_{r'}(t) - \bE \left[ X_{r'}(t) \right] \right) \right] \\
	&= \sum_{r \in I \cap (I + (k,l))} a_r a_{r-(k,l)} \bE \left[ \left( X_r(t) - \bE \left[ X_r(t) \right] \right)^2 \right] \\
	&= d_t^2 \sum_{r \in I \cap (I + (k,l))} a_r a_{r-(k,l)}.
\end{align*}

Finally, by \cite[Exercise 2.4.4]{Tao2012},  the empirical spectral measure $L_N(t)(dx)$ of $R^N(t)$ converges to the same limit $\mu_t$ almost surely. The proof is concluded. 
\end{proof}

\section {High-dimensional limit for Wishart-type matrices} \label{sec:Wishart}

\subsection{Real case}

Recall that $\{X_{ij}(t)\}_{i,j \ge 1}$ are i.i.d. copies of $X_t$ which is the solution to \eqref{SDE}.  Let \[\widehat{U}^N(t) = \left( \widehat{U}_{ij}^N(t) \right)_{1 \le i \le p,\, 1\le j \le N}\] be a $p \times N$ matrix with entries $\widehat{U}_{ij}^N(t) = X_{ij}(t) - \bE \left[ X_{ij}(t) \right]$. Here, $p = p(N)$ is a positive integer that depends on $N$. Let
\begin{align} \label{matrix entries Wishart}
	U^N(t) = \dfrac{1}{N} \widehat{U}^N(t) \widehat{U}^N(t)^\intercal
\end{align}
be a $p \times p$ symmetric matrix with $p$ eigenvalues $\lambda_1^N(t) \le \cdots \le \lambda_p^N(t)$, and
\begin{align*}
	L_N(t)(dx) = \dfrac{1}{p} \sum_{i=1}^p \delta_{\lambda_i^N(t)}(dx)
\end{align*}
be the empirical spectral measure  of $U^N(t)$.

\begin{theorem} \label{Thm-LED Wishart}
Suppose that one of the following conditions holds,
\begin{enumerate}
	\item [(i)] Conditions in Theorem \ref{Thm-Holder1} hold and $\bE[|X_0|^2] < \infty$;
	\item [(ii)] Conditions in Theorem \ref{Thm-Holder2} hold and $\bE [|X_0|^4] < \infty$.
\end{enumerate}
Assume that there exists a positive function $\varphi(x) \in C^1(\bR)$ with bounded derivative, such that $\lim_{|x| \rightarrow \infty} \varphi(x) = + \infty$ and
\begin{align*}
	\sup_{N \in \bN} \langle \varphi, L_N(0) \rangle \le C_0,
\end{align*}
for some positive constant $C_0$ almost surely. Furthermore, assume that there exists a positive constant $c$, such that $p/N \rightarrow c$ as $N \rightarrow \infty$.
	
Then for any $T > 0$, $\bE [|X_t|^2] < \infty$ for $t \in [0,T]$, and the sequence  $\{L_N(t), t\in[0,T]\}_{N \in \bN}$ converges in probability to $\{\mu_t, t\in[0,T]\}$ in $C([0,T], \mathbf P(\bR))$, where  $\mu_t(dx)=\mu_{MP}(c,d_t)(dx)$ with $\mu_{MP}$ given in \eqref{eq:mp-law}.
\end{theorem}

\begin{proof}
 Noting that $\bE[|X_{ij}(t)|^2]$ exists finitely  for all $1 \le i \le p, 1 \le j \le N$, we have that $\widehat U_{ij}^N(t)$ has mean 0 and finite second moment  $d_t^2:=\bE[|\widehat U_{ij}^N(t)|^2]$.
Then by Lemma \ref{lemma-MP law}, for any $t \in [0,T]$, almost surely, the empirical distribution
\begin{align} \label{eq-6.0}
	L_N(t)(dx) \rightarrow \mu_{MP}(c,d_t)(dx)
\end{align}
weakly as $N\to \infty$. Thus, it remains to obtain the tightness of $\{L_N(t), t\in[0,T]\}_{N \in \bN}$ in the space $C([0,T],\mathbf P(\bR))$.

Recalled that $\e \in (0,H-1/2)$, by Theorem \ref{Thm-Holder1} and Theorem \ref{Thm-Holder2},   we have that $|X_{ij}(t) - X_{ij}(s)| \le \xi_{ij} |t-s|^{H-\e}$, where $\{\xi_{ij}\}_{1 \le i \le p, 1\le j \le N}$ are i.i.d. copies of $\xi$ with $\bE[|\xi|^4]<\infty$. Thus,
\begin{align} \label{eq-6.1'}
	 \left| \widehat{U}_{ij}^N(t) - \widehat{U}_{ij}^N(s) \right|
	 &\le \left| X_{ij}(t) - X_{ij}(s) \right| + \left| \bE \left[ X_{ij}(t) \right] - \bE \left[ X_{ij}(s) \right] \right| \nonumber \\
	 &\le \left| X_{ij}(t) - X_{ij}(s) \right| + \bE \left[ \left| X_{ij}(t) - X_{ij}(s) \right| \right] \nonumber \\
	 &\le \left( \xi_{ij} + \bE \left[ \xi_{ij} \right]\right) (t-s)^{H-\e}.
\end{align}
Hence, by \eqref{eq-6.1'}, for $0\le s<t\le T$,
\begin{align} \label{eq-6.2'}
	&\quad \left| \widehat{U}_{ik}^N(t) \widehat{U}_{jk}^N(t) - \widehat{U}_{ik}^N(s) \widehat{U}_{jk}^N(s) \right|^2 \nonumber \\
	&\le 2\left| \widehat{U}_{ik}^N(t) - \widehat{U}_{ik}^N(s) \right|^2 \left| \widehat{U}_{jk}^N(t) \right|^2 + 2\left| \widehat{U}_{ik}^N(s) \right|^2 \left| \widehat{U}_{jk}^N(t) - \widehat{U}_{jk}^N(s) \right|^2 \nonumber \\
	&\le 2 \left( \xi_{ik} + \bE \left[ \xi_{ik} \right]\right)^2 \left( \left| \widehat{U}_{jk}^N(0) \right| + \left( \xi_{jk} + \bE \left[ \xi_{jk} \right]\right) T^{H-\e} \right)^2 (t-s)^{2H-2\e} \nonumber \\
	&\quad + 2 \left( \xi_{jk} + \bE \left[ \xi_{jk} \right]\right)^2 \left( \left| \widehat{U}_{ik}^N(0) \right| + \left( \xi_{ik} + \bE \left[ \xi_{ik} \right]\right) T^{H-\e} \right)^2 (t-s)^{2H-2\e} \nonumber \\
	&= E_{H,T,\e}^{(i,j;k)} (t-s)^{2H-2\e}.
\end{align}
Here, $E_{H,T,\e}^{(i,j;k)}$ is a positive random variable that has finite second moment which is given by
\begin{align*}
	E_{H,T,\e}^{(i,j;k)} &= 2 \left( \xi_{ik} + \bE \left[ \xi_{ik} \right]\right)^2 \left( \left| \widehat{U}_{jk}^N(0) \right| + \left( \xi_{jk} + \bE \left[ \xi_{jk} \right]\right) T^{H-\e} \right)^2 \\
	&\quad + 2 \left( \xi_{jk} + \bE \left[ \xi_{jk} \right]\right)^2 \left( \left| \widehat{U}_{ik}^N(0) \right| + \left( \xi_{ik} + \bE \left[ \xi_{ik} \right]\right) T^{H-\e} \right)^2.
\end{align*}
Let, for $i\neq j$,
\begin{align*}
	E_1 = \bE \left[ E_{H,T,\e}^{(i,j;k)} \right]
	= 4 \bE \left[ \left( \xi + \bE [\xi] \right)^2 \right] \bE \left[ \left( \left| X_0 - \bE [X_0] \right| + \left( \xi + \bE [\xi]\right) T^{H-\e} \right)^2 \right],
\end{align*}
 and for $i=j$,
\begin{align*}
	E_2 = \bE \left[ E_{H,T,\e}^{(i,i;k)} \right]
	= 4 \bE \left[ \left( \xi + \bE [\xi] \right)^2 \left( \left| X_0 - \bE [X_0] \right| + \left( \xi + \bE [\xi]\right) T^{H-\e} \right)^2 \right]. 
\end{align*}
Then $E_1,E_2$ are two positive numbers depending only on $(H,T,\e)$.

Without loss of generality, we assume that
 $\frac{p-1}{N-1} \le c+1$.  Recall that the entries of $\widehat{U}^N(t)$ are independent. Using the Cauchy-Schwarz inequality twice, the mean value theorem, Lemma \ref{lemma-Hoffman ineq}, and \eqref{eq-6.2'}, we can obtain
\begin{align} \label{eq-6.3'}
	&\quad \bE \left[ \left| \langle f, L_N(t) \rangle - \langle f, L_N(s) \rangle \right|^2 \right] = \bE \left[ \left| \dfrac{1}{p} \sum_{i=1}^p f(\lambda_i^N(t)) - f(\lambda_i^N(s)) \right|^2 \right] \nonumber \\
	&\le \dfrac{1}{p} \bE \left[ \sum_{i=1}^p \left| f(\lambda_i^N(t)) - f(\lambda_i^N(s)) \right|^2 \right] \le \dfrac{\|f'\|_{\infty}^2}{p} \bE \left[ \sum_{i=1}^p \left| \lambda_i^N(t) - \lambda_i^N(s) \right|^2 \right] \nonumber \\
	&\le \dfrac{\|f'\|_{\infty}^2}{p} \bE \left[ \sum_{i,j=1}^p \left( U_{ij}^N(t) - U_{ij}^N(s) \right)^2 \right] \nonumber \\
	&= \dfrac{\|f'\|_{\infty}^2}{p} \sum_{i \not= j}^p \bE \left[ \left( \dfrac{1}{N} \sum_{k=1}^N \left[\widehat{U}_{ik}^N(t) \widehat{U}_{jk}^N(t) - \widehat{U}_{ik}^N(s) \widehat{U}_{jk}^N(s) \right]\right)^2 \right] \nonumber \\
	&\qquad\quad  + \dfrac{\|f'\|_{\infty}^2}{p} \sum_{i=1}^p \bE \left[ \left( \dfrac{1}{N} \sum_{k=1}^N\left[ \widehat{U}_{ik}^N(t)^2 - \widehat{U}_{ik}^N(s)^2\right] \right)^2 \right] \nonumber \\
	&\le \dfrac{\|f'\|_{\infty}^2}{pN^2} \sum_{i \neq j}^p \sum_{k=1}^N \bE \left[ \left( \widehat{U}_{ik}^N(t) \widehat{U}_{jk}^N(t) - \widehat{U}_{ik}^N(s) \widehat{U}_{jk}^N(s) \right)^2 \right] \nonumber \\
	&\quad\qquad + \dfrac{\|f'\|_{\infty}^2}{pN} \sum_{i=1}^p \sum_{k=1}^N \bE \left[ \left( \widehat{U}_{ik}^N(t)^2 - \widehat{U}_{ik}^N(s)^2 \right)^2 \right] \nonumber \\
	&\le \dfrac{\|f'\|_{\infty}^2}{pN^2} \sum_{i \neq j}^p \sum_{k=1}^N \bE \left[ E_{H,T}^{(i,j;k)} \right] (t-s)^{2H-2\e} +  \dfrac{\|f'\|_{\infty}^2}{pN} \sum_{i=1}^p \sum_{k=1}^N \bE \left[ E_{H,T}^{(i,i;k)} \right] (t-s)^{2H-2\e} \nonumber \\
	&= \dfrac{\|f'\|_{\infty}^2 (p-1)}{N} E_1 (t-s)^{2H-2\e} +  \|f'\|_{\infty}^2E_2 (t-s)^{2H-2\e} \nonumber \\
	&\le \left( (c+1) E_1 + E_2 \right) \|f'\|_{\infty}^2 (t-s)^{2H-2\e}
\end{align}
for any $f \in C^1(\bR)$ with bounded derivative. Hence, by  Proposition \ref{remark:b1} and \eqref{eq-6.0}, we can conclude that the sequence  $\{L_N(t), t\in[0,T]\}_{N \in \bN}$ converges in law to $\{\mu_t=\mu_{MP}(c,d_t), t\in[0,T]\}$. Finally,  noting that the limit measure $\{\mu_t,t\in[0,T]\}$ is  deterministic,  the convergence in law actually coincides with the convergence in probability. 

The proof is concluded. 
\end{proof}

\begin{remark} \label{remark-4}
In contrast, the convergences of the empirical  measure-valued processes obtained in Theorem \ref{Thm-tightness} and other subsequent results in Section \ref{sec:hdl} are  {\em almost-sure} convergence, which is stronger than the {\em in-probability} convergence obtained in Theorem \ref{Thm-LED Wishart}.

In section \ref{sec:hdl}, we construct a compact set in $C([0,T],
\mathbf P(\bR))$ and show that the sequence $\{L_N(t), t \in [0,T]\}$
is in that compact set {\em almost surely}. However, in the Wishart
case,   we are not able to get an estimation analogous to
\eqref{eq-4.8} which is the key ingredient to get the almost-sure convergence,  due to the lack of the independence for the upper triangular entries. Instead, we  obtain the tightness  on $C([0,T], \mathbf P(\bR))$ for $\{L_N(t), t \in [0,T]\}$ thanks to Proposition \ref{Thm-1'},  and then the convergence {\em in law} follows consequently.  


\end{remark}

\begin{remark} \label{remark-7}
Let $\sigma(x) = 1$, $b(x) = 0$ and $X_0 = 0$, then the solution to \eqref{SDE} is the fractional Brownian motion $X_t = B_t^H$. Then we have the convergence in law of the empirical spectral measures towards the scaled  Marchenko-Pastur law $\mu_{MP}(c,t^H)(dx)$, which  recovers the results obtained in \cite{Pardo2017}.
\end{remark}

\begin{remark} \label{remark-2}
Let $\widetilde{Y}^N(t) = \left( X_{ij}(t) \right)_{1 \le i \le p,1\le j \le N}$. Then under the conditions in Theorem \ref{Thm-LED Wishart}, the sequence of empirical  measures of the eigenvalues of $\frac{1}{N} \widetilde{Y}^N(t) \widetilde{Y}^N(t)^\intercal$ converges in probability to $\mu_{MP}(c,d_t)(dx)$ in $C([0,T],\mathbf P(\bR))$.  Indeed, by the Lidskii inequality in \cite[Exercise 1.3.22 (ii)]{Tao2012}, we have
\begin{align*}
	\left| F_{\frac{1}{N} \widetilde{Y}^N(t) \widetilde{Y}^N(t)^\intercal}(x) - F_{\frac{1}{N} \widehat{U}^N(t) \widehat{U}^N(t)^\intercal}(x) \right|
	\le \dfrac1{N} \br\left(\frac{1}{N} \widetilde{Y}^N(t) \widetilde{Y}^N(t)^\intercal - \frac{1}{N} \widehat{U}^N(t) \widehat{U}^N(t)^\intercal\right),
\end{align*}
where $F_A(x)$ is the number of the eigenvalues of $A$ that are smaller than $x$. Noting that the rank of 
\begin{align*}
	&\quad \dfrac{1}{N} \widetilde{Y}^N(t) \widetilde{Y}^N(t)^\intercal - \dfrac{1}{N} \widehat{U}^N(t) \widehat{U}^N(t)^\intercal \\
	&= \frac{1}{N} \left( \widehat{U}^N(t) + m_t E_N \right) \left( \widehat{U}^N(t)^\intercal + m_t E_N \right) - \frac{1}{N} \widehat{U}^N(t) \widehat{U}^N(t)^\intercal \\
	&= \dfrac{m_t}{N} E_N \widehat{U}^N(t)^\intercal + \dfrac{m_t}{N} \widehat{U}^N(t) E_N + m_t^2 E_N,
\end{align*}
is at most $3$ for all $t \in [0,T]$,  the convergence in probability of $\{L_N(t)\}_{N \in \bN}$ towards $\mu_{MP}(c,d_t)(dx)$ implies that the empirical spectral measures of $\frac{1}{N} \widetilde{Y}^N(t) \widetilde{Y}^N(t)^\intercal$ converges to the same limit in probability.
\end{remark}

\begin{remark} \label{remark-2'}
The Stieltjes transform $G_t(z)$ of the limiting measure $\mu_t$ is
\begin{align*}
	G_t(z) = \int \dfrac{\mu_t(dx)}{z-x}
	= \int \dfrac{p_{MP}(c,d_t)(x)}{z-x} dx
	= \int \dfrac{p_{MP}(c,1)(x)}{z - d_t^2 x} dx,
\end{align*}
where $p_{MP}(c,d_t)(x)$ is the probability density of the
Marchenko-Pastur distribution
$\mu_{MP}(c,d_t)$ given in~\eqref{eq:mp-law}. 
Assuming that the variance $d_t^2$ of the solution $X_t$ is continuously differentiable on $(0,T)$, we have
\begin{align} \label{eq-derivative-MP}
\partial_t G_t(z)
	&= \partial_t \int \dfrac{p_{MP}(c,1)(x)}{z - d_t^2 x} dx
	= (d_t^2)' \int \dfrac{x p_{MP}(c,1)(x)}{(z - d_t^2 x)^2} dx \nonumber \\
	&= \dfrac{(d_t^2)'}{d_t^2} \int \left( \dfrac{z}{(z - d_t^2 x)^2} - \dfrac{1}{z - d_t^2 x} \right) p_{MP}(c,1)(x) dx \nonumber \\
	&= - \dfrac{(d_t^2)'}{d_t^2} \left( z \partial_z G_t(z) + G_t(z) \right).
\end{align}

On the other hand, [Bai and Silverstein, Lemma 3.11] and some computation yield
\begin{align*}
	cz d_t^2 G_t(z)^2 = G_t(z) \left( z - d_t^2 (1-c) \right) -1.
\end{align*}
Taking partial derivative with respect to $z$,  we have
\begin{align} \label{eq-Stransform-MP}
	c d_t^2 G_t(z)^2 + 2 cz d_t^2 G_t(z) \partial_z G_t(z) = G_t(z) + \partial_z G_t(z) \left( z - d_t^2 (1-c) \right).
\end{align}
Therefore, by \eqref{eq-derivative-MP} and \eqref{eq-Stransform-MP}, we have
\begin{align}\label{eq-5.8'}
\partial_t G_t(z)
	= - (d_t^2)' \left( c G_t(z)^2 + 2cz G_t(z) \partial_z G_t(z) + (1-c) \partial_z G_t(z) \right).
\end{align}
\end{remark}

\subsection{Complex case} \label{subsec:Wishart complex}

Recall that $Z=(Z^{(1)}, Z^{(2)})$ is the solution to   \eqref{SDE-complex}.  Let $\widehat{W}^N(t) = \left( \widehat{W}_{ij}^N(t) \right)_{1 \le i \le p,1\le j \le N}$ be a $p \times N$ matrix with entries $\widehat{W}_{ij}^N(t) = Z_{ij}(t) - \bE \left[ Z_{ij}(t) \right]$, where $Z_{ij}$ are i.i.d. copies of $Z^{(1)}+\iota Z^{(2)}$ and $p = p(N)$ is a positive integer depending on $N$. Let
\begin{align} \label{matrix entries Wishart-complex}
	W^N(t) = \dfrac{1}{N} \widehat{W}^N(t) \widehat{W}^N(t)^*
\end{align}
be a $p \times p$ symmetric matrix with eigenvalue empirical measure $L_N(t)(dx)$.

\begin{theorem} \label{Thm-LED Wishart complex}
Suppose that the coefficient functions $\tilde{\sigma}$, $\tilde{b}$ have bounded derivatives which are H\"{o}lder continuous of order greater than $1/(H-\e)-1$. Besides, assume that one of the following conditions holds,
\begin{enumerate}
	\item [(a)] $\|(\tilde{\sigma}_x, \tilde{\sigma}_y)\|_{L^{\infty}(\bR^2)} + \|(\tilde{b}_x, \tilde b_y)\|_{L^{\infty}(\bR^2)} > 0$, $\bE[\|Z_0\|^4] < \infty$.
	\item [(b)] $\tilde{\sigma}$ and $\tilde{b}$ are bounded and $\bE[\|Z_0\|^2] < \infty$.
\end{enumerate}
Moreover, suppose that there exists a positive function $\varphi(x) \in C^1(\bR)$ with bounded derivative, such that $\lim_{|x| \rightarrow \infty} \varphi(x) = + \infty$ and
\begin{align*}
	\sup_{N \in \bN} \langle \varphi, L_N(0) \rangle \le C_0,
\end{align*}
for some positive constant $C_0$ almost surely. Furthermore, suppose that there exists a positive constant $c$, such that $p/N \rightarrow c$ as $N \rightarrow \infty$.
	
Then for any $T > 0$, $\bE [\|Z_t\|^2] < \infty$, and the sequence  $\{L_N(t),[0,T]\}_{N \in \bN}$ converges  in probability to $\mu_{MP}(c,d_Z(t))(dx)$ in $C([0,T],\mathbf P(\bR))$.
\end{theorem}

\begin{proof}
The proof is similar to the  proofs of Theorem \ref{Thm-LED Wishart} and Theorem \ref{Thm-LED complex}, which is sketched below.
	
From the proof of Theorem \ref{Thm-LED complex}, we can obtain the finiteness of the mean $m_Z(t)$ and $d_Z^2(t)$. Analogous to \eqref{eq-6.0},  by using Lemma \ref{lemma-MP law complex}, we have the almost-sure convergence
\begin{align} \label{eq-7.5}
	L_N(t)(dx) \rightarrow \mu_{MP}(c,d_Z(t))(dx).
\end{align}
	
Note that the estimation \eqref{eq-7.3} in the proof Theorem \ref{Thm-LED complex} is still valid. Similar to the estimation \eqref{eq-6.2'} and \eqref{eq-6.3'} in the proof of Theorem \ref{Thm-LED Wishart}, we can obtain
\begin{align*}
	\bE \left[ \left| \langle f, L_N(t) \rangle - \langle f, L_N(s) \rangle \right|^2 \right] \le C \|f'\|_{\infty}^2 (t-s)^{2H-2\e}.
\end{align*}
 Then following the argument at the end of the proof of Theorem \ref{Thm-LED Wishart}, we can obtain the tightness of the sequence $\{L_N(t)\}_{N \in \bN}$, which implies the convergence in distribution and hence the convergence in probability, with the  deterministic limit  given in \eqref{eq-7.5}.
\end{proof}

\begin{remark} \label{remark-3}
Let $\widetilde{W}^N(t) = \left( Z_{ij}(t) \right)_{1 \le i \le p,1\le j \le N}$. Then under the conditions in Theorem \ref{Thm-LED Wishart complex}, the sequence of empirical spectral measures of $\frac{1}{N} \widetilde{W}^N(t) \widetilde{W}^N(t)^\intercal$ converges  in probability to $\mu_{MP}(c,d_Z)(dx)$ in $C([0,T],\mathbf P(\bR))$. 
\end{remark}

\begin{remark} \label{remark-4'}
	Similar to Remark \ref{remark-2'}, the Stieltjes transform of the limit measure $\mu_t$ satisfies the differential equation \eqref{eq-5.8'} with $d_t$ replaced by $d_Z(t)$.
\end{remark}

\appendix

\section{Preliminaries on (random) matrices}\label{sec:pre-matrix}

The  following is the Hoffman-Wielandt lemma, which can be found in \cite[Lemma 2.1.19]{Anderson2010}, see also \cite{Tao2012}.
\begin{lemma}[Hoffman-Wielandt] \label{lemma-Hoffman ineq}
Let $A = (A_{ij})_{1 \le i, j \le N}$ and $B = (B_{ij})_{1 \le i, j \le N}$ be $N \times N$ Hermitian matrices, with ordered eigenvalues $\lambda_1^A \le \lambda_2^A \le \ldots \le \lambda_N^A$ and $\lambda_1^B \le \lambda_2^B \le \ldots \le \lambda_N^B$. Then
\begin{align*}
	\sum_{i=1}^N \left| \lambda_i^A - \lambda_i^B \right|^2
	\le \bT\left[(A-B)(A-B)^*\right]
	= \sum_{i,j=1}^N \left| A_{ij} - B_{ij} \right|^2.
\end{align*}
\end{lemma}

The next two lemmas are the famous Wigner semi-circle law for the real case and complex case respectively (see, e.g.,  \cite{Tao2012}).
\begin{lemma}\label{lemma-Wigner semicircle}
Let $M_N$ be the top left $N \times N$ minors of an infinite Wigner matrix $(\xi_{ij})_{i,j\ge1}$, which is symmetric, the upper-triangular entries $\xi_{ij}, i>j$ are i.i.d. real random variables with mean zero and unit variance, and the diagonal entries $\xi_{ii}$ are i.i.d. real variables, independent of the upper-triangular entries, with bounded mean and variance. Then the empirical spectral distributions $\mu_{M_N/\sqrt{N}}$ converge almost surely to the Wigner semicircular distribution
\begin{align*}
	\mu_{sc}(dx) = \dfrac{\sqrt{4-x^2}}{2\pi} 1_{[-2,2]}(x) dx.
\end{align*}
\end{lemma}

\begin{lemma} \label{lemma-Wigner semicircle-complex}
Let $M_N$ be the top left $N \times N$ minors of an infinite complex Wigner matrix $(\xi_{ij})_{i,j\ge1}$, which is Hermitian, the upper-triangular entries $\xi_{ij}, i>j$ are i.i.d. complex random variables with mean zero and unit variance, and the diagonal entries $\xi_{ii}$ are i.i.d. real variables, independent of the upper-triangular entries, with bounded mean and variance. Then the conclusion of Lemma \ref{lemma-Wigner semicircle} holds.
\end{lemma}

The next two lemmas  concern the celebrated Marchenko-Pastur law which  was introduced in \cite{Bai2010}.

\begin{lemma} \label{lemma-MP law}
Let $X_N$ be the top left $p(N) \times N$ minors of an infinite random matrix, whose entries are i.i.d. real random variable with mean zero and variance $\sigma^2$. Here, $p(N)$ is a positive integer such that $p(N)/N \rightarrow c \in (0,\infty)$ as $N\rightarrow \infty$. Then the empirical distribution of the eigenvalues of the $p \times p$ matrix
\begin{align*}
	Y^N = \dfrac{1}{N} X_NX_N^\intercal
\end{align*}
converges weakly  to the Marchenko-Pastur distribution 
\begin{align}
	\mu_{MP}(c,\sigma)(dx)
	&= \dfrac{1}{2 \pi \sigma^2 cx} \sqrt{\left(\sigma^2 (1+\sqrt{c})^2 - x\right) \left(x - \sigma^2 (1-\sqrt{c})^2\right)} 1_{[\sigma^2(1-\sqrt{c})^2, \sigma^2(1+\sqrt{c})^2]} (x) dx \notag\\
	&\quad + \left( 1 - \dfrac{1}{c} \right) \delta_0(x) dx 1_{[c>1]}\,,\label{eq:mp-law}
\end{align}
almost surely, where $\delta_0$ is the point mass at the origin.
\end{lemma}

\begin{lemma} \label{lemma-MP law complex}
Let $X_N$ be the top left $p(N) \times N$ minors of an infinite random matrix, whose entries are i.i.d. complex random variable with mean zero and variance $\sigma^2$. Here, $p(N)$ is a positive integer such that $p(N)/N \rightarrow c \in (0,\infty)$ as $N\rightarrow \infty$. Then the empirical distribution of the $p \times p$ matrix
\begin{align*}
	Y^N = \dfrac{1}{N} X_NX_N^*
\end{align*}
converges almost surely to the Marchenko-Pastur distribution $\mu_{MP}(c,\sigma)(dx)$ described in Lemma \ref{lemma-MP law}.
\end{lemma}

The following result characterizes  the limiting empirical  spectral distribution of the symmetric random matrix with correlated entries, which is a  direct corollary of \cite[Theorem 3]{Banna2015}.

\begin{lemma} \label{lemma-correlated entries}
Let $(\xi_{i,j})_{(i,j) \in \bZ^2}$ be an array of i.i.d. real-valued random variables with finite second moment. Let $I$ be a finite subset of $\bZ^2$, $\{a_r:r \in I\}$ be a family of constants and
\begin{align*}
	X_{i,j} = \sum_{r \in I} a_r \xi_{(i,j) + r}, \quad 1 \le i \le j.
\end{align*}
Suppose that $\bE[X_{0,0}] = \sum_{r \in I} a_r \bE[\xi_{0,0}] = 0$. Denote $\gamma_{k,l} = \gamma_{l,k} = \bE [X_{0,0} X_{k,l}]$ for all $k\le l$. Let $X^N = \left( X_{i,j}^N \right)_{1 \le i, j \le N}$ be a symmetric matrix with entries $X_{i,j}^N = X_{i,j}/\sqrt{N}$ for $1 \le i \le j \le N$. Then the  empirical spectral measure of $X^N$ converges to a nonrandom probability measure $\mu_c$ with Stieltjes transform $S_c(z) = \int_0^1 h(x,z) dx$, where $h(x,z)$ is the solution to the equation
\begin{align*}
	h(x,z) = \left( -z + \int_0^1 f(x,y) h(y,z) dy \right)^{-1} ~\text{ with }~ f(x,y) = \sum_{k,l\in \bZ} \gamma_{k,l} e^{-2\pi i(kx+ly)}.
\end{align*}
\end{lemma}

\section{Tightness criterions for probability measures on  $C([0,T], \mathbf P(\bR))$} \label{subset:tightness argument}

In this section, we collect some lemmas used in the proofs, and then we  provide two tightness criterions for probability measures on $C([0,T], \mathbf P(\bR))$ (Theorems \ref{Thm-0} and \ref{Thm-0'}).  We also provide sufficient conditions for tightness which can be verified by computing moments (Propositions \ref{Thm-1}, \ref{Thm-1'} and  \ref{remark:b1}).

Note that there has been fruitful literature on tightness of probability measures on a Skorohod space $\mathcal D([0,T], E)$, where $E$ is a completely regular topological space.  We refer the interested reader to \cite{Mitoma, jaku, EK, Dawson1993, Perkins, KX, Kouritzin} and the references therein.  Theorem \ref{Thm-0} is a direct consequence of  Jakubowski's criterion \cite{jaku} (see also e.g,  \cite[Theorem 3.6.4]{Dawson1993} and \cite{Sun2011} for the statement of the criterion), noting that $C([0,T], \mathbf P(\bR))$ is a closed subset of $\mathcal D([0,T], \mathbf P(\bR))$.  Theorem \ref{Thm-0'} might be also well-known in the literature of tightness criterion for probability measures, but we could not find a reference addressing this explicitly. For both Theorems \ref{Thm-0} and \ref{Thm-0'}, we include self-contained proofs for the reader's convenience.

 Recall that  $\mathbf P(\bR)$ is the set of  probability measures on $\bR$ endowed with its weak topology, and that $C([0,T], \mathbf P(\bR))$ is the space of continuous probability-measure-valued processes, both of which are Polish spaces. Denote by $C_0(\bR)$ the set of continuous functions on $\bR$  vanishing at infinity, which is also  a Polish space. Also the space $\tilde{\mathbf P}(\bR)$ of sub-probabilities on $\bR$ endowed with its vague topology is a Polish space (see, e.g., \cite[Theorem 4.2]{Kallenberg}), and so is $C([0,T], \tilde{\mathbf P}(\bR)).$

   Let's also recall some basic facts for probability measures on a Polish space $X$ (see, e.g., \cite{Billingsley} for details). Denote by $\mathbf P(X)$ the set of probability measures on  $(X,\mathcal B_X)$ where $\mathcal B_X$ is the Borel $\sigma$-field on the Polish space $X$.  Let $\Pi\subset\mathbf P(X)$ be a family of probability measures on $X$. The family $\Pi$ is called tight if for every $\e\in(0,1)$,  there exists a compact set $K^\e\subset X$ such that $P(K)>1-\e$ for all $P\in \Pi$. The family $\Pi$ is called relatively compact if every sequence of elements of $\Pi$ contains a weakly convergent subsequence. The Prokhorov's theorem guarantees the equivalence between tightness and relatively compactness. Also note that a sequence  $P_n\in \mathbf P(X)$ converges weakly  to $P\in \mathbf P(X)$ if and only if  $P_n$ converges to $P$ in the Polish space $\mathbf P(X)$.

The following lemma (see, e.g., \cite[Theorem 3.2.14]{Durrett2019}) provides a method to obtain tightness for a set of probability measures.
\begin{lemma} \label{Lemma-B3'}
Let $\mathbb I$ be an index set.  If there is a non-negative function $\varphi$ so that $\varphi(x) \rightarrow \infty$ as $|x| \rightarrow \infty$ and
\begin{align*}
	 \sup_{i \in \mathbb I} \int_{\bR} \varphi(x) \mu_n(dx) < \infty,
\end{align*}
then the family of probability measures $\{\mu_i\}_{i\in \mathbb I}$ is tight.
\end{lemma}
 Based on the above tightness criterion, one can construct compact subsets of $\mathbf P(\bR)$:
\begin{lemma} \label{Lemma-B3}
A set of the form 
 \[K=\left\{\mu\in \mathbf P(\bR): \int_{\bR} \varphi(x) \mu(dx)\le M \right\}\]
 is compact in $\mathbf P(\bR)$, where $M$ is a positive constant and $\varphi(x)$ is given in Lemma \ref{Lemma-B3'}.
\end{lemma}
\begin{proof}
By Lemma \ref{Lemma-B3'} and Prokhorov's theorem (see, e.g., \cite[Theorems 5.1 and 5.2]{Billingsley}) which claims that a  subset $A$ of $\mathbf P(\bR)$  is tight if and only if the closure of $A$ is compact,  it suffices to show that $K$ is a closed set in $\mathbf P(\bR)$, which is easy to verify. 
\end{proof}

By the Arzela-Ascoli Theorem, we have the following lemma to construct compact sets in $C([0,T],\mathbf P(\bR))$.
\begin{lemma} \label{Lemma-B4}
\begin{align*}
	\mathfrak C = \bigcap_{n \in \bN} \left\{ g \in C([0,T], \bR): \sup_{t,s \in [0,T], |t-s| \le \eta_n} |g(t) - g(s)| \le \e_n, \sup_{t \in [0,T]} |g(t)| \le M \right\},
\end{align*}
is  compact in $C([0,T], \mathbf P(\bR))$, where $M$ is a positive constant and $\{\e_n, n \in \bN\}$ and $\{\eta_n, n \in \bN\}$ are two sequences of positive  numbers going to zero as $n$ goes to infinity.
\end{lemma}

The following lemma  (\cite[Lemma 4.3.13]{Anderson2010}) provides an approach to construct compact subsets in $C([0,T], \mathbf P(\bR))$. It will be used in the proof of Theorem \ref{Thm-0}. 
\begin{lemma} \label{lemma in Anderson}
Let $K$ be a compact subset of  $\mathbf P(\bR)$, let $\{f_i\}_{i \in \bN}$ be a sequence of bounded continuous functions that is dense in $C_0(\bR)$, and let $\{C_i\}_{i\in\bN}$ be a family of compact subsets of $C([0,T], \bR)$.  Then the set
\begin{align*}
	\mathfrak C = \Big\{ \mu\in C([0,T], \mathbf P(\bR)):  \mu_t \in K,  \forall t \in [0,T] \Big\} \cap \bigcap_{i \in \bN} \left\{ t \rightarrow \mu_t(f_i) \in C_i \right\}
\end{align*}
is a compact subset of $C([0,T], \mathbf P(\bR))$.
\end{lemma}

The following lemma, which constructs compacts subsets in $C([0,T], \tilde {P}(\bR))$, will play a critical role in the proof of Theorem \ref{Thm-0'}.

\begin{lemma} \label{lem:compact}
Let $\{f_i\}_{i \in\bN}$ be a countable dense subset of $C_0(\bR)$,  and let $\{C_i\}_{i\in\bN}$ be a family of compact subsets of $C([0,T], \bR)$. Then the set
\begin{align*}
	\mathcal K =  \bigcap_{i \in\bN} \left\{ t \rightarrow \mu_t(f_i) \in C_i \right\}
\end{align*}
is a compact subset of $C([0,T], \tilde{\mathbf P}(\bR))$. 
\end{lemma}

\begin{proof}
The proof is  similar to the proof of Lemma 4.3.13 in \cite{Anderson2010},  which is provided here for the reader's convenience. 

Noting that $\mathcal K$ is a closed subset of $C([0,T],\tilde{\mathbf P}(\bR))$ which is a Polish space, it suffices to prove that $\mathcal K$ is sequentially compact.  

Take a sequence $\mu^{(n)} \in \mathcal{K}$. Then the functions $t \rightarrow \mu_t^{(n)}(f_i) \in C_i$ for $\in \bN$. Let $D$ be a countable dense subset of $[0,T]$. Note that for each $t\in D$, $\mu_t^{(n)}$ has a subsequence that converges vaguely, i.e., converges in the Polish space $\tilde {\mathbf P}(\bR)$. Then by the diagonal procedure and the compactness of $C_i$, we can find a subsequence $\mu^{\phi(n)}$ such that $t \rightarrow \mu_t^{\phi(n)}(f_i)$ converges in $C_i$ for all $i\in \bN$ and $\mu_t^{\phi(n)}$ converges in $\tilde{\mathbf P}(\bR)$ for all $t \in D$, as $n$ tends to infinity. Denoting $\varphi_i(t):=\lim_{n \rightarrow \infty} \mu_t^{\phi(n)}(f_i)$ for all $i \in \mathbb{N}, t \in [0,T]$ and $\mu_t:=\lim_{n \rightarrow \infty} \mu_t^{\phi(n)}$ for all $t \in D$, then $\varphi_i\in C_i\subset C([0,T], \bR)$ for $i\in \bN$ and $\mu_t\in \tilde{\mathbf P}(\bR)$ for $t\in D$. The vague convergence of the measures $\mu_t^{\phi(n)}$ for $t\in D$ implies that $\varphi_i(t) = \mu_t(f_i)$ for all $i \in \mathbb{N}$ and $t \in D$. Noting that $\{f_i\}_{i \in \mathbb{N}}$ is dense in $C_0(\mathbb{R})$, $D$ is dense in $[0,T]$, $\varphi_i(t)$ is continuous, one can extend the family  $\{\mu_t, t \in D\}$ of sub-probability measures uniquely to a sub-probability-measure-valued process $\{\nu_t , t\in [0,T]\}\in C([0,T], \tilde{\mathbf P}(\mathbb{R}))$ such that  $\lim_{n \rightarrow \infty} \mu^{\phi(n)} = \nu$. This shows that $\mathcal K$ is sequentially compact in $C([0,T], \tilde{\mathbf P}(\bR))$, and the proof is completed. 
\end{proof}

  
   
Throughout the rest of the section, let  $T$ be a fixed positive number, and let $\{\mu^{(n)}\}_{n\in \bN}:=\{\mu_t^{(n)}, t \in [0,T]\}_{n \in \bN}\subset C([0,T], \mathbf P(\bR))$ be a sequence of continuous probability-measure-valued stochastic processes.  We assume the following conditions on $\{\mu^{(n)}\}_{n\in\bN}$, which will be used in Theorems \ref{Thm-0} and \ref{Thm-0'}. 
\begin{enumerate}
	\item[(I)]  For any $\e>0$, there exists a compact set $K^\e$ in $\mathbf P(\bR)$, such that for all $n\in\bN$, \[\mathbb P\left(\mu_t^{(n)}\in K^\e, \forall t\in [0,T]\right)\ge1-\e.\] 
	
	\item[(I')] For each $t\in[0,T]$, the family of $\mathbf P(\bR)$-valued random elements $\{ \mu_{t}^{(n)}: \Omega \rightarrow \mathbf P(\bR) \}_{n \in \bN}$ is tight. 
	
	\item[(II)] There exists a countable dense subset $\{f_i\}_{i \in \bN}$ of $C_0(\bR)$, such that for each $f_i$, $\{ \langle f_i, \mu_t^{(n)} \rangle, t \in [0,T] \}_{n \in \bN}$ is tight on $C([0,T], \bR)$.
\end{enumerate}

 We also list some conditions that imply the above conditions. 

Let $\varphi(x)$ a be nonnegative function such that $\lim\limits_{|x|\to\infty}\varphi(x) =\infty$. 
\begin{enumerate}
\item[(A)]
\begin{equation}\label{e:uniform-bound}
\sup_{n\in\bN}\bE\left[\sup_{t\in[0,T]} \langle \varphi,\mu_t^{(n)}\rangle\right]<\infty.
\end{equation}
 
\item[(A')] For each $t\in [0,T]$,  the family  $\{\langle \varphi, \mu_t^{(n)} \rangle\}_{n \in \bN}$ of random variables is tight.
\item[(A'')]   For each $t\in [0,T]$, 
 \begin{equation}\label{e:B4}
 \sup_{n\in\bN}\bE \left[ \left| \langle \varphi, \mu_t^{(n)} \rangle \right|^{\alpha} \right]<\infty,
 \end{equation}
for some $\alpha>0$,
\item[(B)] There exists a countable  dense subset  $\{f_i\}_{i\in\bN}$ of $C_0(\bR)$, such that  there exist positive constants $\alpha$, $\beta$,
\begin{align}\label{eq-kol-bd}
	\bE \left[ \left| \langle f, \mu_t^{(n)} \rangle - \langle f, \mu_s^{(n)} \rangle \right|^{1+\alpha} \right] \le C_{f,T} |t-s|^{1+\beta}, \quad \forall t,s \in [0,T], \forall n \in \bN,
\end{align}
for all $f \in \{f_i\}_{i\ge0}$, where $C_{f,T}$ is a constant depending only on $f$ and $T$. 

 \end{enumerate}

\begin{lemma}\label{lem:suf-con}
 (A) $\Rightarrow$ (I); (A'')$\Rightarrow$ (A')$\Rightarrow (I')$;  (B)$\Rightarrow$(II).
\end{lemma}
\begin{proof}
By \eqref{e:uniform-bound}, let $M:=\sup_{n\in\bN}\bE\left[\sup_{t\in[0,T]} \langle \varphi,\mu_t^{(n)}\rangle\right]<\infty$. For  any $\e>0$, choose $M_\e=M/\e$. The set $K^\e:=\{\mu\in\mathbf P(\bR): \langle\varphi, \mu\rangle \le M_\e\}$ is a compact subset of $\mathbf P(\bR)$ by Lemma \ref{Lemma-B3}. Then, by Markov inequality we have, for $n\in\bN,$
\begin{align*}
&\mathbb P\left(\mu_t^{(n)}\notin K^\e, \text{ for some } t\in[0,T]\right)= \mathbb P\left(\sup_{t\in[0,T]}\langle \varphi, \mu_t^{(n)}\rangle >M_\e \right)\le M/M_\e=\e.
\end{align*}
 Thus, (A) $\Rightarrow$ (I).

(A'')$\Rightarrow$(A') follows directly from Lemma \ref{Lemma-B3'}. Now we show (A')$\Rightarrow$(I').
  Fix an arbitrary $t\in [0,T]$. For  any $\e>0$, due to the tightness of $\{\varphi, \mu_t^{(n)}\}_{n\in \bN}$, one can find a positive constant $N_\e$ such that   for all $n \in \bN$,
\begin{align*}
	\bP \left( \langle \varphi, \mu_{t}^{(n)} \rangle \le  N_\e \right) >1-\e.
\end{align*}
This implies that for all $n \in \bN$,
\[\bP \left(\mu_{t}^{(n)}\in C^\e\right) >1-\e,\]
where $C^\e=\{\nu\in \mathbf P(\bR): \langle \varphi, \nu\rangle\le N_\e\}$ is a compact subset of $\mathbf P(\bR)$ by Lemma \ref{Lemma-B3}. Therefore, $\{\mu_t^{(n)}\}_{n\in \bN}$ is tight, and hence (A')$\Rightarrow$(I').

Finally, (B)$\Rightarrow$(II) follows directly from the Kolomogorov tightness criterion (see, e.g., \cite[Theorem 4.2 and Theorem 4.3]{Ikeda1981}).
\end{proof}

The following is Jakubowski' tightness criterion for probability measures on $C([0,T], \mathbf P(\bR))$. 

\begin{theorem} \label{Thm-0}
 Assume that conditions (I) and (II) are satisfied. Then the set $\{ \mu_{t}^{(n)}, t \in [0,T] \}_{n \in \bN}$ induces a tight family of probability measures on $C([0,T], \mathbf P(\bR))$.
\end{theorem}

\begin{proof} 
By condition (II), for any $\e > 0$, there exist compact subsets $ C_i^\e$ of $C([0,T], \bR)$ for $i\in\bN$, such that  for each $i\in\mathbb N$,   for all $n \in \bN$,
\begin{align*}
	\bP \left( \langle f_i, \mu_t^{(n)} \rangle \in C_i^\e \right) \ge 1 - \e/2^i,
\end{align*}
and hence,  for all $n \in \bN$,
\begin{align} \label{eq-0.3}
	\bP \left( \bigcap_{i \in \bN} \left\{ \langle  f_i, \mu_t^{(n)} \rangle \in C_i^\e \right\} \right)
	\ge 1 - \sum_{i \in\bN} \bP \left( \left\{  \langle f, \mu_t^{(n)} \rangle \in C_i^\e \right\}^\complement \right)
	\ge 1 - \e.
\end{align}

By Lemma \ref{lemma in Anderson}, the set 
\begin{align*}
	\mathfrak C(\e) =  \Big\{ \mu\in C([0,T], \mathbf P(\bR)): \mu_{t} \in K^\e, \forall t\in [0,t]\Big\} \cap \bigcap_{i \ge 0} \left\{ t \rightarrow  \langle f_i, \mu_t\rangle\in C_i^\e \right\},
\end{align*}
is compact in $C([0,T], \mathbf P(\bR))$ for any $\e > 0$. By condition (I) and \eqref{eq-0.3}, we have for all $n \in \bN$,
\begin{align*}
	\bP \left( \mu^{(n)} \in \mathfrak C(\e)\right) \ge 1-2\e.
\end{align*}
This  implies the tightness of $\{\mu^{(n)}\}_{n \in \bN}$  on $C([0,T], \mathbf P(\bR))$. The proof is concluded. 
\end{proof}

\begin{remark}\label{remarkB2}
Following the proof of \cite[Theorem 3.1]{jaku}, one can easily show that conditions (I) and (II)  are also  necessary conditions for  tightness of probability measures on $C([0,T], \mathbf{P} (\bR)).$

\end{remark}

The criterion in Theorem \ref{Thm-0} can be verified by computing moments:

\begin{proposition} \label{Thm-1}
 Assume that conditions (A) and  (B) are satisfied. Then the set $\{\mu_t^{(n)}, t \in [0,T]\}_{n \in \bN}$ induces a tight family of probability measures on $C([0,T], \mathbf P(\bR))$.
\end{proposition}

\begin{proof}
The desired result follows directly from Theorem \ref{Thm-0} and Lemma \ref{lem:suf-con}.
\end{proof}

In general situations,  it might not be easy to check  condition (I) or  (A).  Below, we provide another tightness criterion which weakens condition (I).

\begin{theorem} \label{Thm-0'}
 Assume conditions (I') and (II) are satisfied. 
Then the set $\{ \mu_{t}^{(n)}, t \in [0,T] \}_{n \in \bN}$ induces a tight family of probability measures on $C([0,T], \mathbf P(\bR))$.
\end{theorem}

\begin{proof} 


By condition (II), we can choose the same compact subsets $C_i^\e$ of $C([0,T], \bR)$  for $i\in\bN$ as in the proof of Theorem \ref{Thm-0}, and hence \eqref{eq-0.3} still holds.  By Lemma \ref{lem:compact}, the set 
\begin{align*}
	\mathcal K(\e) =   \bigcap_{i \in \bN} \left\{ t \rightarrow  \langle f_i, \mu_t\rangle\in C_i^\e \right\},
\end{align*}
is compact in $C([0,T], \tilde{\mathbf P}(\bR))$ for any $\e > 0$. By \eqref{eq-0.3}, we have, for all $n \in \bN$,
\begin{align*}
	\bP \left( \mu^{(n)} \in \mathcal K(\e)\right) \ge 1-\e.
\end{align*}

 This  implies the tightness of $\{\mu^{(n)}\}_{n\in\bN}$ on $C([0,T], \tilde{\mathbf P}(\bR))$.  Therefore,   for any subsequence of $\mu^{(n)}$, by Prokhorov's theorem, there exists a subsequence $\mu^{(n_k)}$ which converge weakly to some $\nu=\{\nu_t, t\in[0,T]\}\in C([0,T], \tilde{\mathbf P}(\bR))$ which is a continuous sub-probability-measure-valued process. Thus, for each $t\in[0,T]$,  the sequence $\mu_t^{(n_k)}$ (as  $\tilde{\mathbf P}(\bR)$-valued random elements) converges weakly to $\nu_t\in\tilde{\mathbf P}(\bR)$.  This together with the tightness of $\{\mu^{(n)}_t\}_{n\in\bN}$ (as $\mathbf P(\bR)$-valued random elements) in condition (I') implies that $\nu_t\in \mathbf P(\bR)$ and hence $\nu\in C([0,T], \mathbf P(\bR))$. Therefore, Prokhorov's theorem implies that $\{\mu^{(n)}\}_{n\in\bN}$ is tight on $C([0,T], \mathbf P(\bR))$. The proof is concluded. 
\end{proof}

\begin{remark}\label{remarkB2'}
Noting that condition (I) implies condition (I'), then by Remark \ref{remarkB2}, conditions (I') and (II) are also necessary conditions for  tightness of probability measures on $C([0,T], \mathbf{P} (\bR))$. 
\end{remark}

Similarly, we can justify the criterion in Theorem \ref{Thm-0'} by computing moments. The two Propositions below are direct consequences of Theorem \ref{Thm-0'} and Lemma \ref{lem:suf-con}. 
\begin{proposition} \label{Thm-1'}
 Assume that conditions (A') and (B) are satisfied. Then the set $\{\mu_t^{(n)}, t \in [0,T]\}_{n \in \bN}$ induces a tight family of probability measures on $C([0,T], \mathbf P(\bR))$.
\end{proposition}

\begin{proposition} \label{remark:b1}
 Assume that conditions (A'') and (B) are satisfied. Then the set $\{\mu_t^{(n)}, t \in [0,T]\}_{n \in \bN}$ induces a tight family of probability measures on $C([0,T], \mathbf P(\bR))$.
\end{proposition}

%
%

\medskip

{\bf Acknowledgment:}
We would like to thank Rongfeng Sun for reminding us of the Jakubowski’s tightness criterion. J. Yao is partially supported by HKSAR-RGC-Grant GRF-17307319.

\bibliographystyle{alphaabbr}
\bibliography{Reference,WishartProcess}

\end{document}